\documentclass[reqno]{amsart}

\usepackage{enumerate}
\usepackage[colorlinks=true]{hyperref}
\hypersetup{citecolor=blue}

\numberwithin{equation}{section}

\newtheorem{thm}{Theorem}[section]
\newtheorem{lem}[thm]{Lemma}
\newtheorem{prop}[thm]{Proposition}
\newtheorem{cor}[thm]{Corollary}
\newtheorem{defn}[thm]{Definition}
\theoremstyle{definition}
\newtheorem{rem}[thm]{Remark}

\newcommand\R{{\mathbb R}}

\newcommand\N{{\mathbb N}}

\newcommand\Tma{T_{\mathrm{max}}}
\newcommand\Comp{{\mathrm{c}}}
\newcommand\Eqdef{\stackrel{\text{\tiny def}}{=}}
\newcommand\Ens{{\mathcal E}}
\newcommand\Map{{\mathcal F}}

\newcommand\dist{{\mathrm{d}}}

\newcommand\Cz{{C_0(\R^N )}}

\newcommand\Czo{{C_0(\Omega )}}

\newcommand\Md{{\mathcal M}}
\newcommand\esup{\mathop{\mathrm{ess\,sup}}}
\newcommand\Cbu{{C_{\mathrm{b, u}}(\R^N )}}

\newcommand\Loc{{\mathrm{loc}}}

\newcommand\goto{\mathop{\longrightarrow}}

\newcommand\MScN[1]{\href{http://www.ams.org/mathscinet-getitem?mr=#1}{\nolinkurl{(#1)}}}
\newcommand\DOI[1]{\href{http://dx.doi.org/#1}{(doi: \nolinkurl{#1})}}
\newcommand\LINK[1]{\href{#1}{(link: \nolinkurl{#1})}}

\newcommand\DI{u_0 }
\newcommand\DIb{v_0 }
\newcommand\DIbd{w_0 }

\newcommand\MUU{\mu }

\newcommand\CA[1]{A_{#1} }

\begin{document}

\title{Perturbations of self-similar solutions}

\def\shorttitle{Perturbations of self-similar solutions}

\author[T. Cazenave]{Thierry Cazenave$^1$}
\email{\href{mailto:thierry.cazenave@sorbonne-universite.fr}{thierry.cazenave@sorbonne-universite.fr}}

\author[F. Dickstein]{Fl\'avio Dickstein$^{1,2}$}
\email{\href{mailto:flavio@labma.ufrj.br}{flavio@labma.ufrj.br}}

\author[I. Naumkin]{Ivan Naumkin$^3$}
\email{\href{mailto:ivan.naumkin@iimas.unam.mx}{ivan.naumkin@iimas.unam.mx}}

\author[F. B.~Weissler]{Fred B.~Weissler$^4$}
\email{\href{mailto:weissler@math.univ-paris13.fr}{weissler@math.univ-paris13.fr}}

\address{$^1$Sorbonne Universit\'e \& CNRS, Laboratoire Jacques-Louis Lions,
B.C. 187, 4 place Jussieu, 75252 Paris Cedex 05, France}

\address{$^2$Instituto de Matem\'atica, Universidade Federal do Rio de Janeiro, Caixa Postal 68530, 21944--970 Rio de Janeiro, R.J., Brazil}

\address{$^3$Departamento de F\'\i sica Matem\'atica,
Instituto de Investigaciones en Matem\'aticas Aplicadas y en Sistemas, 
Universidad Nacional Aut\'onoma de M\'exico, 
Apartado Postal 20-126, M\'exico DF 01000, M\'exico}

\address{$^4$Universit\'e Paris 13, Sorbonne Paris  Cit\'e, LAGA CNRS UMR 7539, 99 Avenue J.-B. Cl\'e\-ment, F-93430 Villetaneuse, France}



\thanks{Research supported by the ``Brazilian-French Network in Mathematics"}
\thanks{Fl\'avio Dickstein was partially supported by CNPq (Brasil).}
\thanks{This work was prepared while Ivan Naumkin was visiting the Laboratoire J.A. Dieudonn\'e of the Universit\'e de Nice Sophia-Antipolis. 
He thanks the project ERC-2014-CdG 646.650 SingWave for its financial support, and the Laboratoire J.A. Dieudonn\'e for its kind hospitality. }

\begin{abstract}
We consider the nonlinear heat equation $u_t = \Delta u +  |u|^\alpha u$ with $\alpha >0$, either on $\R^N $, $N\ge 1$, or on a bounded domain with Dirichlet boundary conditions. 
We prove that in the Sobolev subcritical case $(N-2) \alpha <4$, for every $\mu \in \R$, if the initial value $\DI $ satisfies $\DI(x) = \mu  |x-x_0|^{-\frac {2} {\alpha }}$ in a neighborhood of some $x_0\in \Omega $ and is bounded outside that neighborhood, then there exist infinitely many solutions of the heat equation with the initial condition $u(0)= \DI$.  The proof uses a fixed-point argument to construct perturbations of self-similar solutions with initial value
$\mu  |x-x_0|^{-\frac {2} {\alpha }}$ on $\R^N $.

Moreover, if $\mu \ge \mu _0$ for a certain $ \mu _0( N, \alpha )\ge 0$, and $\DI\ge 0$, then there is no nonnegative local solution of the heat equation with the initial condition $u(0)= \DI$, but there are infinitely many sign-changing solutions.

\end{abstract}

\maketitle


\section{Introduction}

In this paper we study the local well-posedness for the nonlinear heat equation
\begin{equation} \label{Inlh} 
\begin{cases} 
u_t = \Delta u +  |u|^\alpha u \\
u ( 0, \cdot ) = \DI (\cdot )
\end{cases} 
\end{equation} 
where $u= u(t,x)$, $t\ge 0$, $x \in \Omega $ where $\Omega $ is a domain in $\R^N $ (possibly $\Omega =\R^N $), and $\alpha >0$.
In the case where $\Omega \not = \R^N $, we impose Dirichlet boundary conditions. 
There is already a vast literature devoted to this topic, and it is well known that this problem is locally well-posed in various function spaces, for example in $\Czo $, in $L^p (\Omega ) $ for $p\ge 1$, $p>\frac {N\alpha } {2}$, and in certain Sobolev and Besov spaces of both positive and negative order. 
On the other hand, this problem is not locally well-posed in $L^p (\Omega ) $ if $\alpha >\frac {2} {N}$ and $1\le p < \frac {N\alpha } 2{}$. We refer the reader to the book~\cite{QuittnerS} as a general reference to this subject. 

The present paper is concerned with the situation where the problem~\eqref{Inlh} is not locally well posed.
For example, regular initial values can yield multiple solutions which are continuous into $L^p (\Omega )$, $1\le p<\frac {N\alpha } {2}$, see~\cite{HarauxW, Baras, Terraneo, MatosT}.  Also, it was observed in~\cite{Weissler2, Weissler4, BarasP, GalaktionovV, QuittnerS, LaisterRS, LaisterRSVL, FujishimaI}, for these same values of $p$, that there are nonnegative  $\DI \in L^p (\Omega )$ for which there is no local-in-time nonnegative solution.
This phenomenon has been taken as evidence of non-local-well-posedness. Indeed, in all the spaces where~\eqref{Inlh} is locally well-posed, nonnegative initial values yield nonnegative solutions. 
Thus if there is a nonnegative initial value for which there is no nonnegative solution, one might think that there is no solution at all. Our recent paper~\cite{CDNW1} shows this to be incorrect. 
More precisely, in the case $\Omega =\R^N $ and $0<\alpha < \frac {4} {N-2}$, then for the initial value $\DI (x)= \mu  |x|^{- \frac {2} {\alpha }}$ with $\mu >0$ sufficiently large, there is no local nonnegative solution of~\eqref{Inlh}, but there exist nonetheless infinitely many global solutions of~\eqref{Inlh} which change sign.
In other words, local well-posedness fails, not because of nonexistence, but because of nonuniqueness. 

The purpose of this paper is to extend the results of~\cite{CDNW1} to a much broader context. 
In this previous work, the initial values studied are all homogeneous, of the form $\DI (x)= \mu  |x|^{- \frac {2} {\alpha }}$,  and the resulting solutions are all self-similar. 
In particular, these initial values are not in any space $L^p (\R^N ) $.
Here, we consider initial values, both on $\R^N $ and on a bounded domain $\Omega $, which exhibit a point singularity of the same form $\DI (x)= \mu  |x -x_0 |^{- \frac {2} {\alpha }}$ near some point $x_0\in \Omega $, but which have a general behavior away from $x_0$. 
For a wide class of such initial values, we prove that there are infinitely many local solutions of~\eqref{Inlh}. In some of these cases, the initial value is nonnegative and there is no local nonnegative solution of~\eqref{Inlh}.
In the case $\alpha >\frac {2} {N}$, this includes initial values that belong to $L^p (\Omega ) $ for all $1\le p<\frac {N\alpha } {2}$.

In order to state  our nonuniqueness results, we introduce some notation. 
Given a domain $\Omega \subset \R^N $, we denote by $ (e^{t \Delta _\Omega }) _{ t\ge 0 }$ the heat semigroup on $\Czo $, the completion of $C^\infty _\Comp (\Omega )$ in $L^\infty  (\Omega ) $, and by $G _\Omega  (t,x,y)$  the associated heat kernel. 
In the particular case $\Omega =\R^N $, we let  $ (e^{t \Delta _{\R^N } }) _{ t\ge 0 } = (e^{t \Delta  }) _{ t\ge 0 }$ and $G _{\R^N }  (t,x,y) = G (t,x,y) = (4\pi t)^{-\frac {N} {2}} e^{- \frac { |x-y|^2 } {4t}}$.
Note that $e^{t \Delta _\Omega }$ can be extended to $L^1 (\Omega ) +L^\infty  (\Omega ) $, and to measurable functions $\DI : \Omega \to [0,\infty )$ by setting
\begin{equation*} 
e^{t \Delta _\Omega } \DI (x)= \int _\Omega G_\Omega (t,x,y ) \DI (y) \,dy
\end{equation*} 
for $x\in \Omega $. In the latter case, the right-hand side is the integral of a nonnegative, measurable function, and therefore well defined, finite or infinite.
We also denote by $\Cbu$ the Banach space of  uniformly continuous, bounded functions $\R^N \to \R$, equipped with the sup norm. 

As consequences of the main results proved in the body of the paper, we have the following two results, respectively on $\R^N $ and on a bounded domain $\Omega $.

\begin{thm} \label{ePR1} 
Let $0< \alpha <\frac {4} {N-2}$, $\MUU \in \R$ and suppose $\DI \in L^\infty _\Loc (\R^N \setminus \{0\} ) $ 
is such that
\begin{enumerate}[{\rm (i)}] 

\item $\DI (x) =  \MUU   |x|^{-\frac {2} {\alpha }} $ in a neighborhood of $0$.

\item $\displaystyle \DI (x) \goto \noindent  _{  |x| \to \infty } 0$.

\end{enumerate} 
It follows that there exists a sequence $(u^m) _{ m\ge 1 }$ of distinct sign-changing solutions of~\eqref{Inlh},   $u^m \in C((0,T^m), \Cz) $, with $T^m >0$, and the initial condition is satisfied in the sense that $u^m(t) \to \DI$ in $L^p   _\Loc ( \R^N \setminus \{0\})$ for all $1\le p<\infty $. 

Furthermore, if  $\alpha >\frac {2} {N}$,  these solutions satisfy the integral equation
\begin{equation} \label{fNZ2} 
u(t) = e^{t\Delta } \DI + \int _0^t e^{ (t-s) \Delta }  |u(s)|^\alpha u(s)\, ds
\end{equation} 
where  the integral is convergent  in $L^r (\R^N ) + L^\infty  (\R^N ) $ for all $1\le  r < \frac {N\alpha } {2}$ and each term is in $C((0,T^m ), \Cz )$.  Moreover, $u^m - e^{t \Delta } \DI \in  C([0,T^m ), L^r (\R^N ) + L^\infty  (\R^N ) )$ for all $r \ge 1$ such that $\frac {N\alpha } {2(\alpha +1) } < r < \frac {N\alpha } {2}$. 
In addition, if $\DI\in C( \R^N \setminus \{ 0\} )$ then $u^m \in  C([0,T^m ), L^r (\R^N ) + L^\infty  (\R^N ) )$ for the same values of $r$.

Moreover, there exists $\MUU_0\ge 0$ such that if  $\MUU>\MUU_0$ and $\DI \ge 0$, then equation~\eqref{Inlh} has no local nonnnegative solution.
\end{thm} 

\begin{thm} \label{ePR2} 
Let $0< \alpha <\frac {4} {N-2}$, $\MUU \in \R$, let $\Omega \subset \R^N $ be a bounded, smooth domain and $x_0\in \Omega $. 
Suppose $\DI \in L^\infty _\Loc ( \Omega  \setminus \{ x_0\} ) $  is such that
\begin{enumerate}[{\rm (i)}] 

\item $\DI (x) =  \MUU   |x - x_0 |^{-\frac {2} {\alpha }} $ for $ |x -x_0 |<\delta $, where $\delta >0$.

\item $\DI \in  L^\infty  (\Omega \cap \{  |x - x_0|>\delta  \}) $.

\end{enumerate} 
It follows that there exists a sequence $(u^m) _{ m\ge 1 }$ of distinct, sign-changing solutions of
\begin{equation} \label{NLHD} 
\begin{cases} 
u_t= \Delta u  +  |u |^\alpha u \\
u _{ |\partial \Omega  }=0 \\
u  (0) = \DI
\end{cases} 
\end{equation} 
where $u^m \in C((0,T^m), \Czo) $, with $T^m >0$, and the initial condition is satisfied in the sense that $u^m(t) \to \DI$ in $L^\infty    ( \Omega  \setminus \{  |x - x_0|<\varepsilon \})$ for all $\varepsilon >0$. 

Furthermore, if  $\alpha >\frac {2} {N}$,  these solutions satisfy the integral equation 
\begin{equation} \label{NLHID} 
u(t) = e^{t\Delta _\Omega } \DI + \int _0^t e^{(t-s) \Delta _\Omega }  |u(s)|^\alpha u(s) \, ds
\end{equation} 
where  the integral is convergent in $L^r (\Omega  ) )$ for all $1\le  r < \frac {N\alpha } {2}$, and  each term is in $C((0,T^m ), \Czo )$. Moreover, $u^m \in  C([0,T^m ), L^r (\Omega  ) )$ for all $1\le  r < \frac {N\alpha } {2}$.

In addition, there exists $\MUU_0\ge 0$ such that if  $\MUU>\MUU_0$ and $\DI \ge 0$, 
then equation~\eqref{Inlh} has no local nonnnegative solution.
\end{thm}  

The basic method used to prove the above theorems is a perturbation argument. We consider a self-similar solution $U$ with initial value $\mu  |x|^{- \frac {2} {\alpha }}$, known to exist, and look for a solution $u(t,x)$ of the form
\begin{equation*} 
u (t, x) = \Psi (x) U (t, x) + w (t, x),
\end{equation*} 
where $\Psi $ is a cut-off function. 
The purpose of the function $\Psi $ is to ensure that, in the case where $\Omega \not = \R^N $, $u$ and $w$ satisfy the same (Dirichlet) boundary conditions.
The integral equation satisfied by the unknown function $w$ is then solved by using a fixed point argument. 
The difficulty is that the equation satisfied by $w$ contains some terms that are highly singular at $(t,x)= (0,0)$. 
The metric space in which the fixed point argument is carried out is made up of functions which vanish with sufficiently high order at $(0,0)$, so as to balance the singular terms in the equation. The delicate point in the construction of this set is that it must at the same time be stable by the iteration process.

This method has some limitations and the results we obtain are not as strong as what we would like or what we think is true. First, the nature of the fixed-point argument requires that $\DI (x) \equiv  \MUU   |x|^{-\frac {2} {\alpha }} $  in a neighborhood of $0$. On the other hand, we expect that Theorems~\ref{ePR1} and~\ref{ePR2} would be true as stated, but with that condition replaced by the requirement that $\DI  -  \MUU   | \cdot |^{-\frac {2} {\alpha }} \in L^\infty  (\Omega ) $. 
Second, Theorem~\ref{ePR1} does not give a nonuniqueness result in $L^r (\R^N )$ for $r \ge 1$ such that $\frac {N\alpha } {2(\alpha +1) } < r < \frac {N\alpha } {2}$ (see Remark~\ref{eRemNU4} for details). 
However, in the case of a bounded domain (Theorem~\ref{ePR2}), the solutions $u^m$ are all in $C([0,T^m), L^r (\Omega  ) )$ for $1\le r  < \frac {N\alpha } {2}$.
This is a genuine nonuniqueness result in $L^r (\Omega )$.  

The results of the current paper as well as~\cite{CDNW1} call for a reevaluation of the notion of nonexistence of local solutions. For all the nonnegative initial values for which no nonnegative local solution exists, the possibility remains that these initial values give rise to sign-changing solutions.
To our knowledge, there is no example of an initial value for which it is known that there is no local solution of~\eqref{Inlh} or~\eqref{NLHID}, sign-changing or not.

We mention here some recent papers which are related to this work. Under appropriate restrictions on $\alpha $ and $\mu $, in~\cite{FMY}, the authors prove the existence of nonnegative solutions, some global, some non-global, which have the homogeneous initial value $\mu  |x|^{- \frac {2} {\alpha }}$, but which are not self-similar. The more general equation $u_t = \Delta u +  f(u)$ is investigated in~\cite{LaisterRS} and in~\cite{LaisterRSVL}. For $\Omega$ being the whole space or a bounded domain, a full characterisation of the nonnegative functions $f$ for which the equation has a local solution bounded in $L^q(\Omega )$ for all nonnegative initial data $u_0\in L^q(\Omega )$ is stablished in~\cite{LaisterRSVL}.

The rest of the paper is organized as follows. In Section~\ref{sNEPS} we give some conditions under which a nonnegative initial value does not give rise to any local nonnegative solution of either~\eqref{Inlh} or~\eqref{NLHID}. See in particular Corollaries~\ref{eTHF2} and~\ref{eFR3}.
Section~\ref{sPSS} presents the main technical achievement of this article, i.e. the fixed point argument that proves the existence of solutions which are perturbations of a singular solution known already to exist. 

In Sections~\ref{sPSSS} and~\ref{sSCSD} we apply the result of Section~\ref{sPSS} in the case where the known singular solution is in fact a self-similar solution. This gives perturbed solutions, respectively on $\R^N $ (see Theorem~\ref{ePSS1}) and on a bounded domain $\Omega $ (see Theorem~\ref{ePSD1}). 
Combining these two results with the results in~\cite{CDNW1}, we obtain Theorems~\ref{ePR1} and~\ref{ePR2}, respectively.
Finally, we collect in Appendices~\ref{sLHEOM} and~\ref{sLHERN} a few results concerning the regularity of solutions of the heat equation in a form which we need for this paper.  In Appendix~\ref{sRHE}, first on $\R^\N$ and then on a sufficiently smooth bounded domain, we prove that a certain class of solutions of integral equation \eqref{NLHID} are also solutions
of the initial value problem \eqref{Inlh} and give some precise information about their regularity.

\section{Nonexistence of positive solutions} \label{sNEPS} 

We consider the nonlinear heat equation~\eqref{Inlh} on a domain $\Omega \subset \R^N $,
with initial values which are measurable functions $\DI: \Omega \to \R$.
This includes the possibility that $\DI \not \in L^1_\Loc (\Omega )$. 

\begin{defn} \label{eDefSR} 
Let $\Omega $ be a domain of $\R^N $, $\alpha >0$, and let $\DI $ be a measurable function $\Omega \to \R$. Given $T>0$, a regular solution of~\eqref{Inlh} on $(0, T)$ with Dirichlet boundary conditions is a function $u \in C((0,T), \Czo )$ which is a  classical solution  of~\eqref{Inlh} on $(0,T) $, and such that there exists a sequence $t_n \downarrow 0$ such that $u(t_n) \to \DI$ almost everywhere as $n \to \infty $.
\end{defn} 

In the case where $\DI \ge 0$, we also consider solutions of~\eqref{Inlh} that may be singular at positive times.

\begin{defn} \label{eDefSI} 
Let $\Omega $ be a domain of $\R^N $, $\alpha >0$, and let $\DI $ be measurable $\Omega \to [0, \infty )$. Given $T>0$, a nonnegative integral solution of~\eqref{Inlh} on $(0, T)$ with Dirichlet boundary conditions is a measurable function $u: (0, T) \times \Omega \to [0,\infty )$ which satisfies 
\begin{equation} \label{fLoc1} 
u(t,x) = \int _\Omega G_\Omega (t,x,y) \DI (y) \, dy + \int _0^t \int _\Omega G_\Omega (t-s, x, y) ( |u|^\alpha u) (s, y) \, dy \,ds
\end{equation} 
a.e. on $(0,T) \times \Omega $. (The integrands in the right-hand side of~\eqref{fLoc1} are nonnegative, measurable functions, so the integrals are well defined, possibly infinite.)
\end{defn} 

\begin{rem} 
Here are some comments on the above definitions.
\begin{enumerate}[{\rm (i)}] 

\item 
If $u$ is a regular solution of~\eqref{Inlh} on $(0, T)$ in the sense of Definition~\ref{eDefSR}, and if $u\ge 0$, then $u$ need not be an integral solution in the sense of Definition~\ref{eDefSI}.
Indeed, assuming $\alpha <\frac {2} {N}$, we construct a positive solution of~\eqref{Inlh} with $\DI =0$ in the sense of Definition~\ref{eDefSR}, which is not a solution of~\eqref{fLoc1} with $\DI= 0$.
To do this, we recall that if $\alpha < \frac {2} {N}$, then the initial value problem~\eqref{Inlh} is locally well posed in the space of bounded measures. In particular, there exists a positive local solution $u$ for $\DI= \delta  _{ x_0 }$  where $\delta  _{ x_0 }$ is the Dirac measure at $x_0\in \Omega $, which is a solution of the integral equation~\eqref{fLoc1} with $\DI= \delta  _{ x_0 }$. On the other hand, $u(t, x) \to 0$ for all $x\not = x_0$ as $t\to 0$, so that the initial value in the sense of Definition~\ref{eDefSR} is $u(0)= 0$, but $u$ is not a solution of~\eqref{fLoc1} with $\DI= 0$.

\item 
Nonnegative integral solutions of~\eqref{Inlh} on $(0, T)$ in the sense of Definition~\ref{eDefSI} may have a singularity (in space) for all $t\in (0,T)$, so they need not be regular solutions in the sense of Definition~\ref{eDefSR}.
 For instance, if $\alpha >\frac {2} {N-2}$, then $u (t,x)= (  \frac {2} {\alpha }(  N -2 - \frac {2} {\alpha } ) )^{\frac{1}{\alpha}}|x|^{-\frac{2}{\alpha}}$ is a nonnegative integral solution the sense of Definition~\ref{eDefSI}. 
See~\cite[Lemma~7.1]{CDNW1}.
\end{enumerate} 

\end{rem} 

We recall the following result, which is a special case of~\cite[Theorem~1]{Weissler4}.

\begin{prop}  \label{eTHF1} 
Let $\Omega $ be a domain of $\R^N $, $\alpha >0$,  $\DI $ a measurable function $\Omega \to [0, \infty )$, and  $T>0$. If there exists a nonnegative integral solution of~\eqref{Inlh} on $(0, T)$ with Dirichlet boundary conditions (in the sense of Definition~$\ref{eDefSI}$), then 
\begin{equation} \label{fTHF1} 
\sup  _{ 0< t < T } (\alpha t)^{\frac {1} {\alpha }}  \| e^{t \Delta _\Omega } \DI \| _{ L^\infty  } \le 1
\end{equation} 
\end{prop} 

\begin{proof} 
It follows from~\eqref{fLoc1} that $\DI\in L^1_\Loc (\Omega ) $, for if not, $u(t,x) $ would be infinite for every $ ( t,x) \in (0,T) \times \Omega $. In particular, $\DI $ defines a Borel measure on $\Omega $. 
If $\Omega $ is a bounded, smooth domain, then the result follows from Theorem~1 in~\cite{Weissler4}. 
Even though the result in~\cite{Weissler4} is stated for a bounded, smooth domain, the same proof is valid for an arbitrary domain. 
\end{proof} 

\begin{cor}  \label{eTHF2} 
Let $\Omega $ be a domain $\R^N $, $\alpha >0$, and $\DI $ a measurable function $\Omega \to [0, \infty )$. Suppose
\begin{equation} \label{feFR1:1} 
\limsup  _{ t\downarrow 0 } ( \alpha t )^{\frac {1} {\alpha }} \| e^{t\Delta _\Omega } \DI \| _{ L^\infty  } > 1 .
\end{equation} 

\begin{enumerate}[{\rm (i)}] 

\item \label{eTHF2:1} 
There does not exist any nonnegative integral solution of~\eqref{Inlh} with Dirichlet boundary conditions on any interval $(0,T)$ with $T>0$, in the sense of Definition~$\ref{eDefSI}$.

\item \label{eTHF2:2} 
There does not exist any nonnegative regular solution of~\eqref{Inlh} with Dirichlet boundary conditions on any interval $(0,T)$ with $T>0$, in the sense of Definition~$\ref{eDefSR}$.

\item \label{eTHF2:3} 
If $\DI^n = \min\{ \DI, n \}\in L^\infty  (\Omega ) $ for $n\ge 0$, and  if $\Tma ( \DI^n )>0$ is the maximal existence time of the corresponding solution of~\eqref{Inlh} with Dirichlet boundary conditions, then $\Tma ( \DI^n) \to 0$  as $n\to \infty $.

\end{enumerate} 
\end{cor} 

\begin{proof} 
Property~\eqref{eTHF2:1} is an immediate consequence of Proposition~\ref{eTHF1}.

We next prove Property~\eqref{eTHF2:2}. 
Suppose there exist $T>0$ and a nonnegative regular solution $u$ on $(0,T)$, in the sense of Definition~\ref{eDefSR}. Let $0<\delta <\frac {T} {2}$. 
Since $u$ is a classical solution on $(\delta  ,\delta +\frac {T} {2})$, we have
\begin{equation*} 
u(t +\delta  )= e^{t  \Delta _\Omega } u( \delta  ) + \int _0 ^t e^{(t- s ) \Delta _\Omega }  |u(s+\delta )|^\alpha u(s+\delta ) \,ds 
\end{equation*} 
for all $0\le t\le \frac {T} {2} $, and it follows from Proposition~\ref{eTHF1} that
\begin{equation}  \label{feTHF2:1} 
\sup  _{ 0< t < \frac {T} {2} } (\alpha t)^{\frac {1} {\alpha }}  \| e^{t \Delta _\Omega } u(\delta ) \| _{ L^\infty  } \le 1 .
\end{equation} 
We fix $0< t < \frac {T} {2}$ and we let $\delta =t_n$ where $(t_n) _{ n\ge 1 }$ is the sequence in Definition~\ref{eDefSR}. Inequality~\eqref{feTHF2:1} implies that
\begin{equation*} 
 (\alpha t)^{\frac {1} {\alpha }} \int _\Omega G_\Omega (t, x, y) u(t_n, y) \, dy \le 1 .
\end{equation*} 
Letting $n\to \infty $ and applying Fatou's lemma, we deduce (since $u(t_n) \to \DI$ almost everywhere) that
\begin{equation*} 
 (\alpha t)^{\frac {1} {\alpha }} \int _\Omega G_\Omega (t, x, y) \DI ( y) \, dy \le 1 .
\end{equation*} 
Since $t\in (0, \frac {T} {2})$ is arbitrary, this contradicts~\eqref{feFR1:1}. Hence~\eqref{eTHF2:2} is established. 

We finally prove Property~\eqref{eTHF2:3}. We claim that, given any $T>0$, there exists $n_0\ge 1$ such that
\begin{equation} \label{feFR1:2} 
\sup  _{ 0< t < T } (\alpha t)^{\frac {1} {\alpha }}  \| e^{t \Delta _\Omega } \DI ^n \| _{ L^\infty  } > 1
\end{equation} 
for all $n\ge n_0$. Assuming the claim, we deduce from Proposition~\ref{eTHF1} that $\Tma ( \DI^n)\le T$ for all $n\ge n_0$. Since $T>0$ is arbitrary, Property~\eqref{eTHF2:3} follows. 
To prove the claim~\eqref{feFR1:2}, assume by contradiction that there exists a sequence $n_k\to \infty $ such that 
\begin{equation*} 
\sup  _{ 0< t < T } (\alpha t)^{\frac {1} {\alpha }}  \| e^{t \Delta _\Omega } \DI ^{n_k} \| _{ L^\infty  } \le  1 .
\end{equation*} 
Since $\DI^{n_k} \to \DI$ almost everywhere, we deduce from Fatou's lemma that 
\begin{equation*} 
\sup  _{ 0< t < T } (\alpha t)^{\frac {1} {\alpha }}  \| e^{t \Delta _\Omega } \DI \| _{ L^\infty  } \le  1 ,
\end{equation*} 
which contradicts~\eqref{feFR1:1}. Hence~\eqref{eTHF2:3} is established. 
\end{proof} 

The following proposition, derived from a comparison property of~\cite{VDBerg}, gives a sufficient condition, independent of the domain $\Omega $, for inequality~\eqref{feFR1:1} of Corollary~\ref{eTHF2} to hold.
 
\begin{prop} \label{eFR2} 
Let  $\DIb $ be measurable $\R^N  \to [0, \infty )$, and suppose $\DIb \in L^1 ( \{  |x|> \rho \}) + L^\infty   ( \{  |x|> \rho \}) $ for some $\rho >0$ and
\begin{equation} \label{fFBW1} 
\limsup  _{ t\downarrow 0 } (\alpha t)^{\frac {1} {\alpha  } } \int  _{ \R^N    } G(t, 0, y) \DIb  (y)\, dy >1 .
\end{equation} 
Let $\Omega $ be a domain in $\R^N $ with $\{  |x|<\rho   \}\subset  \Omega $ (possibly $\Omega =\R^N $), and let $\DI $ be measurable $\Omega \to [0, \infty )$. 
If $\DI \ge \DIb$ almost everywhere on $\{  |x| <\rho \}$, 
then~\eqref{feFR1:1} holds.
\end{prop} 

\begin{proof} 
We first claim that if $0< \rho < R$ and $\varphi : \R^N \to [0,\infty ) $ is measurable and supported in $\{  |x|\le \rho  \}$, then
\begin{equation} \label{fVDB1} 
 e^{t \Delta  _{ B_R }} ( \varphi _{ |B_R })  \ge  e^{-\frac {\pi ^2 N^2 t} {4(R-\rho ) ^2}}  e^{t \Delta } \varphi   \quad \text{on }  B_\rho 
\end{equation} 
for all $t>0$.
Indeed, it follows from~\cite[Theorem~2]{VDBerg} that 
\begin{equation*} 
G  _{ B_R } (t,x,y) \ge e^{- \frac { |x-y|^2} {4t}} G _{ B_{R-\rho } } (t, 0 , 0 ) =  (4\pi t )^{ \frac {N} {2}} G (t,x,y)  G _{ B_{R-\rho } } (t, 0 , 0 ) 
\end{equation*} 
for all $x,y \in B_\rho $. 
Furthermore, it follows from~\cite[Lemma~9]{VDBerg} that
\begin{equation*} 
G _{ B_{R-\rho }  }(t,  0, 0) \ge (4\pi t )^{- \frac {N} {2}} e^{-\frac {\pi ^2 N^2 t} {4(R-\rho ) ^2}}.
\end{equation*} 
Therefore,
\begin{equation*} 
G  _{ B_R } (t,x,y) \ge e^{-\frac {\pi ^2 N^2 t} {4(R-\rho ) ^2}} G (t,x,y) 
\end{equation*} 
for all $x,y \in B_\rho $,  which implies~\eqref{fVDB1}.

Let now $\DIb $ be as in the statement, and let $R>\rho $ be such that $B_R \subset \Omega $. 
Let 
\begin{equation*} 
\DIbd (x) = \begin{cases} 
\DIb (x) &  |x|<\rho  \\ 0 &  |x|>\rho 
\end{cases} 
\end{equation*} 
and  $ \widetilde{w}_0 = \DIbd \null_{ |\Omega  }$. 
Since  $\DIb \in L^1 ( \{  |x|> \rho \}) + L^\infty   ( \{  |x|> \rho \}) $, we may write
\begin{equation} \label{feFR1:3} 
\DIb = \DIbd + \psi _1 + \psi _2
\end{equation} 
where $\psi _1\in L^1 (\R^N ) $, $\psi _2\in L^\infty (\R^N ) $ and $\psi _1=\psi _2= 0$ on $B_\rho $. 
Next, we have
\begin{equation} \label{feFR1:4} 
(4\pi t)^{\frac {N} {2}} e^{t \Delta }  \psi _1  (0) = \int  _{ \{  |y|>\rho \} } e^{- \frac { |y|^2} {4t}} \psi _1(y) \, dy \le e^{- \frac { \rho ^2} {4t}}  \|\psi _1 \| _{ L^1 }.
\end{equation} 
Furthermore,
\begin{equation} \label{feFR1:5} 
\begin{split} 
(4\pi t)^{\frac {N} {2}} e^{t \Delta }  \psi _2  (0) & = \int  _{ \{  |y|>\rho \} } e^{- \frac { |y|^2} {4t}} \psi _2(y) \, dy \le  \|\psi _2 \| _{ L^\infty  }\int  _{ \{  |y|>\rho \} } e^{- \frac { |y|^2} {4t}}  \, dy \\ & = t^{\frac {N} {2}} \|\psi _2 \| _{ L^\infty  }\int  _{ \{  |y|>\frac {\rho } {\sqrt t} \} } e^{ - \frac { |y|^2} {4}}  \, dy  \le C e^{- \frac { \rho ^2} {8t}} 
\end{split} 
\end{equation} 
and we deduce from~\eqref{feFR1:4} and~\eqref{feFR1:5}  that
\begin{equation}  \label{feFR1:6} 
 (\alpha t)^{\frac {1} {\alpha  } } (  |e^{t \Delta }  \psi _1  (0) | +  |e^{t \Delta }  \psi _2  (0) | ) 
 \goto  _{ t\to 0 }0.
\end{equation} 
Next, we deduce from~\eqref{fVDB1} (with $\varphi = \DIbd $) that
\begin{equation*} 
e^{-\frac {\pi ^2 N^2 t} {4(R-\rho ) ^2}} (\alpha t)^{\frac {1} {\alpha  } }  e^{t\Delta } \DIbd  (0)
\le  (\alpha t)^{\frac {1} {\alpha  } }  e^{t \Delta  _{ B_R } }   \widetilde{w}_0   (0)
\le (\alpha t)^{\frac {1} {\alpha  } }  e^{t \Delta _\Omega }   \widetilde{w}_0   (0) .
\end{equation*} 
Applying~\eqref{feFR1:3}, we deduce that
\begin{equation*} 
(\alpha t)^{\frac {1} {\alpha  } }  e^{t \Delta _\Omega }   \widetilde{w}_0    (0)  \ge e^{-\frac {\pi ^2 N^2 t} {4(R-\rho ) ^2}} (\alpha t)^{\frac {1} {\alpha  } }  e^{t\Delta } \DIb (0) -  (\alpha t)^{\frac {1} {\alpha  } } (  |e^{t \Delta }  \psi _1  (0) | +  |e^{t \Delta }  \psi _2  (0) | ) .
\end{equation*} 
and~\eqref{feFR1:1}  follows from~\eqref{fFBW1} and~\eqref{feFR1:6}.
\end{proof} 

\begin{cor} \label{eFR3} 
Let $\alpha, \gamma , \MUU >0$. 
Suppose at least one of the following three conditions is true:
\begin{gather} 
\gamma \ge N,\quad \MUU >0 , \label{feFR3:b1}  \\
\gamma > \frac {2} {\alpha },\quad \MUU >0 ,  \label{feFR3:b2} \\
\gamma = \frac {2} {\alpha }, \quad \alpha >\frac {2} {N},\quad \MUU > [ \alpha ^{\frac {1} {\alpha  } }   [e^{ \Delta }  |\cdot |^{-\frac {2} {\alpha } }] (0) ]^{-1} .  \label{feFR3:b3}
\end{gather} 
Let $\Omega $ be a domain of $\R^N $, $x_0\in \Omega $, $0<\delta < \dist (x_0, \partial \Omega )$, and $\DI $ a measurable function $\Omega \to [0,\infty )$. If 
\begin{equation} \label{feFR2:3} 
\DI (x) \ge \MUU  |x-x_0|^{- \gamma }  \text{ for }   |x-x_0 |\le \delta 
\end{equation} 
then there is no local nonnegative solution of~\eqref{Inlh}. More precisely,
Properties~\eqref{eTHF2:1}, \eqref{eTHF2:2} and~\eqref{eTHF2:3} of Corollary~$\ref{eTHF2}$ hold. 
\end{cor} 

\begin{proof} 
By space-translation invariance of the equation, we may assume $x_0=0$. 
Applying Corollary~\ref{eTHF2} and Proposition~\ref{eFR2}, we need only show that $\DIb (x)= \MUU  |x|^{- \gamma }$ satisfies~\eqref{fFBW1}. 
This is immediate if $\gamma \ge N$, since $\DIb \not \in L^1_\Loc (\Omega ) $ in this case. 
Therefore, we now suppose 
\begin{equation*} 
\alpha >\frac {2} {N}  \text{ and }  \frac {2} {\alpha } \le \gamma <N.
\end{equation*} 
By scaling invariance 
\begin{equation*} 
e^{t \Delta } \DIb (0) = \MUU t^{- \frac {\gamma } {2}} [e^{ \Delta }  |\cdot |^{-\gamma }] (0) 
\end{equation*} 
so that
\begin{equation*} 
 (\alpha t)^{\frac {1} {\alpha  } }  e^{t \Delta } \DIb (0) = \MUU \alpha ^{\frac {1} {\alpha  } }  t^{\frac {1} {\alpha }- \frac {\gamma } {2}} [e^{ \Delta }  |\cdot |^{-\gamma }] (0) . 
\end{equation*} 
Therefore, if  $\frac {2} {\alpha }< \gamma <N$, then 
\begin{equation*} 
 (\alpha t)^{\frac {1} {\alpha  } }  e^{t \Delta }  \DIb  (0) \goto _{ t\downarrow 0 } \infty ,
\end{equation*} 
so that~\eqref{fFBW1} holds.
Furthermore, if~\eqref{feFR3:b3} is satisfied, then 
\begin{equation*} 
 (\alpha t)^{\frac {1} {\alpha  } }  e^{t \Delta }  \DIb  (0) = \MUU \alpha ^{\frac {1} {\alpha  } }   [e^{ \Delta }  |\cdot |^{-\frac {2} {\alpha } }] (0) >1
\end{equation*} 
so that~\eqref{fFBW1} holds.
\end{proof} 

\begin{rem} \label{eCCM1} 
Here are some comments on the above results.
\begin{enumerate}[{\rm (i)}] 

\item 
Initial values of the form~\eqref{feFR2:3} have been used in~\cite{LaisterRS} to prove nonexistence of nonnegative local solutions. 
(See the proofs of Lemma~4.1 and Theorem~4.1 in~\cite{LaisterRS}.)

\item \label{eCCM1:2} 
The conditions~\eqref{feFR3:b1}, \eqref{feFR3:b2} and~\eqref{feFR3:b3} only depend on the space dimension $N$ and on $\alpha >0$.
In particular, the conditions under which~\eqref{feFR2:3} implies the nonexistence of local nonnegative solutions of~\eqref{Inlh} are independent of the domain $\Omega $ and of $x_0\in \Omega $.

\item \label{eCCM1:1} 
Suppose $\DIb$ satisfies~\eqref{fFBW1} and let $p\ge 1$, $p\ge \frac {N\alpha } {2}$. If $\psi \in L^p (\R^N ) $, then 
\begin{equation} \label{fCCM1} 
\limsup  _{ t\downarrow 0 } (\alpha t)^{\frac {1} {\alpha  } }  e^{t \Delta } ( \DIb +\psi ) (0) >1 .
\end{equation} 
Indeed, since $p\ge \frac {N\alpha } {2}$, the estimate $\| e^{t\Delta } \psi \| _{ L^\infty  }\le t^{-\frac {N} {2p}}  \| \psi \| _{ L^p }$ (and a density argument if $p=\frac {N\alpha } {2}$) implies that  $t^{\frac {1} {\alpha }}  \| e^{t\Delta } \psi \| _{ L^\infty  } \to 0$ as $t\to 0$.

\item  The assumption~\eqref{fFBW1} in Lemma~\ref{eFR2} means that $\DIb $ is sufficiently singular at $x=0$. It does not mean, however, that $\DIb (x) \to \infty $ as $ |x| \to 0$.
Indeed, it can be that $\DIb $ satisfies~\eqref{fFBW1} and $\DIb (x)=0$ for some $x$ arbitrarily close to $0$. Here is such an example. 
Let $\alpha >0$, $\gamma >0$,  $\MUU >0$ be such that~ \eqref{feFR3:b2} or~\eqref{feFR3:b3} holds, and let $\DIbd (x)= \MUU  |x|^{- \gamma }$. 
It follows (see the proof of Corollary~\ref{eFR3}) that $\DIbd$ satisfies~\eqref{fFBW1}. 
Fix $p\ge 1$, $p\ge \frac {N\alpha } {2}$. 
Consider now a sequence $(a_j) _{ j\ge 1 } \subset (0,\infty )$ such that $a_j \downarrow 0$, and let $(b_j) _{ j\ge 1 }$ satisfy  $b_j > a_j > b _{ j+1 }$ for all $j\ge 1$. Setting
\begin{equation*} 
\theta = \DIbd \sum_{ j=1 }^\infty  1 _{ \{ a_j <  |x| < b_j \} }
\end{equation*} 
it is clear that if we choose the sequence $(b_j) _{ j\ge 1 }$ so that $b_j -a_j$ is sufficiently small, then $\theta  \in L^p (\R^N ) $. 
If $\DIb = \DIbd - \theta $, then $\DIb \ge 0$ and it follows from Property~\eqref{eCCM1:1} above that $\DIb $ satisfies~\eqref{fFBW1}. 

\item 
The case of condition~\eqref{feFR3:b2} in Corollary~\ref{eFR3} has already been given in the proof of Theorem~15.3 in~\cite{QuittnerS}, see in particular formula~(15.30). 
\end{enumerate} 
\end{rem} 

\section{Self-similar solutions} \label{sSSS} 

We recall that a self-similar solution of~\eqref{Inlh} is a solution of the form
\begin{equation} \label{fpr1} 
u(t, x) = t^{-\frac {1} {\alpha }} f \Bigl( \frac {x} {\sqrt t} \Bigr),
\end{equation}
where $f  : \R^N \to \R$ is  the profile of the
self-similar solution $u$ given by \eqref{fpr1}.  In order for $u$ given by 
\eqref{fpr1} to be a classical solution of  \eqref{Inlh} for $t > 0$,
the profile $f$ must be of class $C^2$ and satisfy the elliptic equation
\begin{equation} \label{fpr2} 
\Delta f + \frac {1} {2} x\cdot \nabla f + \frac {1} {\alpha } f +  |f|^\alpha f  = 0.
\end{equation} 
A radially symmetric regular self-similar solution of~\eqref{Inlh} is a self-similar solution with a profile $f$ which is of class $C^2$ and radially symmetric. We write, by abuse of notation, $f(r) = f(x)$ where $r= |x|$, so that
$f : [0, \infty) \to \R$ is of class $C^2$, and satisfies the following
initial value ODE problem,
\begin{gather}
\displaystyle f''(r) + \Big(\frac{N-1}{r} +  \frac{r}{2}\Big)f'(r) + \frac{1}{\alpha}f(r) +  |f(r)|^{\alpha}f(r) = 0 \label{fpr3}\\
f(0) = a, \quad f'(0) = 0 \label{fpr3:1}
\end{gather}
for some $a\in \R$.
 If $u$ is a radially symmetric regular self-similar solution of~\eqref{Inlh}, then there exists $\mu \in \R$ such that $r^{\frac {2} {\alpha }} f(r) \to \mu $ as $r\to \infty $ (see~\cite[Theorem~5$'$]{HarauxW}) so that $f\in  L^r (\R^N ) $ for all $r\ge 1$, $r>\frac {N\alpha } {2}$. 
Moreover,   $u$ has the initial value $\mu  |x|^{- \frac {2} {\alpha }}$ in the sense that $u(t,x) \to \mu  |x|^{- \frac {2} {\alpha }}$ as $t\downarrow 0$, uniformly on $\{  |x|>\varepsilon  \}$ for every $\varepsilon >0$.
See~\cite[page~170]{HarauxW}.
It is known that if $\alpha <\frac {4} {N-2}$ and $\mu \in \R$, then for the initial value $\mu  |x|^{- \frac {2} {\alpha }}$ there are infinitely many global, self-similar solutions of~\eqref{Inlh}. 
More precisely, the following holds.

\begin{prop}[\cite{CDNW1,HarauxW, Weissler6, Yanagida}] \label{eRESS1} 
Let $N\ge 1$ and $0< \alpha < \frac {4} {N-2}$ ($0<\alpha <\infty $ if $N= 1,2$). 
Given any $\MUU \in \R$, there exists $m_0\ge 0$ such that for all $m \ge m_0$ there exist at least two different, radially symmetric regular self-similar solutions $U$ of \eqref{Inlh} with initial value $\DI=  \MUU  |x|^{-\frac {2} {\alpha }}$ in the sense that 
\begin{equation} \label{fEP2} 
U(t) \goto _{ t\downarrow 0 } \DI  \text{ in } L^q  ( \{  |x|>\varepsilon  \} )  \text{ for all } \varepsilon >0  \text{ and }  1\le q\le \infty ,  q >\frac {N\alpha } {2} ,
\end{equation} 
and also that
\begin{equation} \label{fEP3} 
U(t) \goto _{ t\downarrow 0 } \DI  \text{ in } L^p (\R^N ) +L^q  (\R^N )   \text{ for all }  1\le p < \frac {N\alpha } {2} <q \le \infty  
\end{equation}
if $\alpha >\frac {2} {N}$, and whose profiles have exactly $m$ zeros. These solutions are such that $U \in C^1((0,\infty ), L^r (\R^N ) )$ and $\Delta U,  | U |^\alpha U\in C((0,\infty ), L^r (\R^N ) )$ for all $1\le r\le \infty $, $r> \frac {N\alpha } {2}$. 

Furthermore, if  $\alpha >\frac {2} {N}$,  the solutions satisfy the integral equation
\begin{equation} \label{fNZ2:b1} 
U(t) =  \MUU  e^{t \Delta  }  |\cdot |^{-\frac {2} {\alpha }}  + \int _0^t e^{ (t-s) \Delta }  |U(s)|^\alpha U(s)\, ds
\end{equation} 
where the integral is norm convergent in $L^r (\R^N ) $ for all $r\ge 1$, $ r >\frac {N\alpha } {2(\alpha +1) }$, and  each term is in $C((0,\infty ), L^r (\R^N ) )$ for all $r> \frac {N\alpha } {2} > 1$. Moreover, the map $t \mapsto U (t) - e^{t \Delta } \DI $ is in $ C([0,\infty ), L^r (\R^N ) )$ for all $r \ge 1$ such that $\frac {N\alpha } {2(\alpha +1) } < r < \frac {N\alpha } {2}$.
\end{prop} 

\begin{proof} 
The case $\MUU\not = 0$ follows from~\cite[Theorems~1.1 and~1.4]{CDNW1}, and the case $\MUU =0$ follows from~\cite{HarauxW, Weissler6, Yanagida}, see~\cite[Remark~1.2~(iv)]{CDNW1}.
In the case $\alpha >\frac {2} {N}$, the convergence of the integral term of~\eqref{fNZ2}  in $L^r (\R^N ) $ for all $r\ge 1$ such that $r \ge 1$ such that $\frac {N\alpha } {2(\alpha +1) } < r < \frac {N\alpha } {2}$ follows from the estimate~(2.29) in~\cite{CDNW1}. 
Furthermore, using formula~\eqref{fpr1} and the commutation relationship of the heat semigroup with space dilations (see e.g.~\cite[formula~(3.1)]{DCDS}), one can verify the following identity
\begin{equation*} 
\int _0^t  \| e^{ (t-s) \Delta } ( |U(s)|^\alpha U(s) ) \| _{ L^r } ds  = t^{ \frac {N} {2r} - \frac {1} {\alpha }} \int _1^\infty  s^{- \frac {N} {2r} + \frac {1} {\alpha } -1 }  \| e^{ (s-1 )\Delta } ( |f|^\alpha f) \| _{ L^r } ds .
\end{equation*} 
The factor $t^{ \frac {N} {2r} - \frac {1} {\alpha }}$ converges to $0$ as $t\to 0$ if and only if $r < \frac {N\alpha } {2}$. We claim that the integral is convergent for all $r\ge 1$ with $r> \frac {N\alpha } {2 (\alpha +1)}$. Indeed, 
$  \| e^{ (s-1 )\Delta } ( |f|^\alpha f) \| _{ L^r }  \le  \| \, |f|^\alpha f \| _{ L^r }  <\infty  $ because $r >  \frac {N\alpha } {2 (\alpha +1)}$, so we need only study the integrability at infinity. We fix any $1\le q\le r$ such that $\frac {N\alpha } {2 (\alpha +1)} < q < \frac {N\alpha } {2}$, and we estimate
$  \| e^{ (s-1 )\Delta } ( |f|^\alpha f) \| _{ L^r } \le (s-1) ^{\frac {N} {2r} - \frac {N} {2q}}  \| f \|^{\alpha +1} _{ L^{q(\alpha +1)} }$,
so that
\begin{equation*} 
\int _2^\infty s^{- \frac {N} {2r} + \frac {1} {\alpha } -1 }  \| e^{ (s-1 )\Delta } ( |f|^\alpha f) \| _{ L^r } ds \le C  \| f \| _{ L^{q(\alpha +1)} } ^{\alpha +1} \int _2 ^\infty  s^{- \frac {N} {2q} + \frac {1} {\alpha } -1 } ds <\infty ,
\end{equation*} 
which completes the proof.
\end{proof} 

\section{Perturbations of singular solutions} \label{sPSS} 

In this section, we establish a perturbation result for solutions of~\eqref{Inlh} with 
singular initial values. More precisely, we start with a known solution $U$ of~\eqref{Inlh}
which is classical for $t>0$ and develops a singularity as $(t,x) \to (0,0)$. For some 
domain $\Omega \subset \R^N $, which may be smaller than the domain where
$U$ is defined but contains $0$, we look
for solutions $u$ of~\eqref{NLHD} of the the form 
\begin{equation} \label{fPT1} 
u (t, x) = \Psi (x) U (t, x) + w (t, x),
\end{equation} 
where $\Psi $ is a smooth function on $\overline\Omega $, identically $1$ near $0$,
and $w$ is bounded.
Such a solution $u$ captures the initial singularity of $U$ near $0$, modified
by some bounded function $w(0,x)$.

Our motivating example is the case where $U$ is a radially symmetric regular self-similar solution, i.e a solution of~\eqref{Inlh} (on $\R^N $) of form $U(t,x)= t^{-\frac {1} {\alpha }} f(\frac { |x|} {\sqrt t})$ with  $f\in C^2 ([0, \infty ))$. It is well known that  $ |f(r)| + (1+r)  |f'(r)|\le C (1+ r^2)^{- \frac {1} {\alpha }}$ for all $r\ge 0$ (see~\cite[Proposition~3.1]{HarauxW}), which implies
\begin{equation} \label{fSSS2} 
 | U(t,x) |^\alpha +  | \nabla U(t,x) |^{\frac {2\alpha } {\alpha +2}} \le \frac {C} {t +  |x|^2} \le C \min \{ t^{-1},  |x|^{-2} \}.
\end{equation} 
More generally, we consider a domain $\Omega \subset \R^N $, $\delta >0$, $T>0$, $\CA1\ge 0$, and a function $U : (0, T) \times \Omega \to \R$ which satisfies
\begin{gather} 
U  \text{ is }C^2  \text{ in }x  \text{ and }C^1  \text{ in } t  \label{ePSSZ1:1b1}\\
U_t - \Delta U =  |U|^\alpha U  \text{ in }(0,T) \times \Omega   \label{ePSSZ1:2}
 \\
 |U (t, x) |^\alpha \le  \CA1  \min \{ t^{-1},  |x|^{-2} \}   \text{ in }(0,T) \times \Omega   \label{ePSSZ1:3} \\
  |\nabla U (t, x) |^\alpha \le  \CA1   \text{ in }(0,T) \times ( \Omega \cap \{ |x|>\delta \} ) .\label{ePSSZ1:4} 
\end{gather} 
If $w$ satisfies \eqref{fPT1}, then the equation for $w$ is
\begin{equation} \label{fPT2} 
w_t - \Delta w = \Md w
\end{equation} 
where
\begin{equation} \label{fPT3}
\Md w= 2\nabla U \cdot \nabla \Psi + U\Delta \Psi + | \Psi U+w|^\alpha (\Psi U+w) - \Psi  |U|^\alpha U.
\end{equation}
The corresponding integral equation is given by
\begin{equation}  \label{fPT4} 
w (t) = e^{t \Delta _\Omega } \DIbd + \int _0^t e^{(t -s) \Delta _\Omega } \Md w (s) \, ds
\end{equation} 
where $\DIbd$ is a prescribed initial value.
If indeed $w$ is a solution of equation~\eqref{fPT4}, it necessarily vanishes on
$\partial \Omega $ (if it is nonempty) for $t>0$, and so in order for $u$ to vanish on $\partial \Omega $, we would need to require that $\Psi$ vanish on $\partial \Omega $.
It turns out that such a condition on $\Psi$ is not needed for the construction
of solutions $w$.

To solve equation~\eqref{fPT4}, it is natural to use a contraction mapping argument. 
A quick look at~\eqref{fPT3} shows that the term $2\nabla U \cdot \nabla \Psi + U\Delta \Psi $ is easily controlled since we suppose that $\Psi \equiv 1$ in a neighborhood of $0$. On the other hand, assuming that $w$ is bounded and that $U$ is singular at $(t,x)= (0,0)$, the remaining term behaves like $  |U|^\alpha w$ near $(t,x)= (0,0)$.
The singularity allowed by~\eqref{fSSS2} corresponds to the ``critical case" and can be treated by the method used in~\cite[Theorem~6.1]{CW}. Unfortunately, this method only works for small data, and in particular for perturbations of small self-similar solutions.

Our solution to this difficulty is to use a contraction mapping argument in a class of functions $w$ that are sufficiently small as $(t,x) \to (0,0)$ so as to balance the singularity of $ |U|^\alpha $. The major problem is then to find a class of functions $w$ that have this behavior near $0$, and which at the same time is preserved by the operator associated with the contraction mapping argument. To achieve this, we consider a class of functions of the form $\{ w\in L^\infty ( (0,T) \times \Omega );\,  |w|\le \Theta   \}$, where the bounded function $\Theta$ tends to $0$ as a power of $t$ as $t \to 0$, on some neighborhood of $x = 0$.  As a consequence of
this method, $\DIbd $ must vanish in a neighborhood of the origin.

Our main result in this section is the following.

\begin{thm} \label{ePSSZ1} 
Let $\Omega $ be a domain of $\R^N $, possibly $\Omega =\R^N $, and  $\delta >0$.
Suppose
\begin{equation}  \label{ePSSZ1:1} 
\{  |x|<\delta  \} \subset \Omega 
\end{equation} 
and let $U $ satisfy~\eqref{ePSSZ1:1b1}--\eqref{ePSSZ1:4}.
Let 
\begin{equation} \label{ePSSZ1:4:b} 
\Psi \in C^2  (\R^N  ) \cap W^{2, \infty }  (\R^N  ), \quad 0\le \Psi \le 1,\quad   \Psi (x)= 1  \text{ on }  \{  |x|<\delta  \}
\end{equation} 
and let 
\begin{equation} \label{ePSSZ1:5} 
\DIbd \in L^\infty  (\Omega ) , \quad \DIbd (x) =0 \text{ a.e. on } \{  |x|<\delta  \}.
\end{equation} 
It follows that there exist $ T >0 $ and $w\in L ^\infty ((0, T ) \times \Omega )$ such that the following properties hold.
\begin{enumerate}[{\rm (i)}] 

\item \label{ePSSZ1:10} $\Md w \in  L^\infty ((0, T ) \times \Omega )$ where $\Md w$ is defined by~\eqref{fPT3}. 

\item \label{ePSSZ1:11} $w$ is a solution of equation~\eqref{fPT4}. 
Moreover, $w$ is the unique solution of~\eqref{fPT4} in the class $\Ens$ defined by~\eqref{fPSDZ14} below, where $\Theta $ is given by~\eqref{fDTTabZ1},  $m$ satisfies~\eqref{fDFM}, and $T$ satisfies~\eqref{fDFM:1}, \eqref{fDFM:2} and~\eqref{fDFM:3}.

\item \label{ePSSZ1:12} $ \| w (t) - e^{t \Delta _\Omega } \DIbd \| _{ L^\infty  }\le Ct $ for $0<t<T $.
\end{enumerate} 
\end{thm} 

\begin{rem} 
For the purposes of Theorem~\ref{ePSSZ1}, the right-hand side of equation~\eqref{fPT4} is interpreted as 
\begin{equation}  \label{fPT4b1} 
 \int _\Omega G_\Omega (t,x,y) \DIbd (y)\, dy +   \int _0^t \int _\Omega  G _\Omega  (t-s ,x,y) \Md w (s, y) \, dy \, ds .
\end{equation} 
The above integrals are well defined for all $(t,x) \in (0,T) \times \Omega $, see Appendix~\ref{sLHEOM}. 
The reason for this particular formulation is to avoid having to impose
regularity conditions on $\Omega $.

\end{rem} 

For the proof of Theorem~\ref{ePSSZ1} we will use several lemmas. 
We first introduce auxiliary functions that will be crucial in the fixed-point argument.

\begin{lem} \label{ePSS2} 
Let $\theta \in C^\infty (\R, \R)$ be nondecreasing and satisfy
\begin{equation} \label{fDFNt} 
\begin{cases} 
\theta (s)= 0 & s\le 1 \\
\theta (s)= 1 & s\ge 2 ,
\end{cases} 
\end{equation} 
let $\delta >0$, and set  $a_j= 2^{-j} \delta$ for $j\ge 0$. The sequence $(\chi _j) _{ j\ge 0 } \subset C^\infty  (\R^N ) $ defined by
\begin{equation}  \label{fePSS1:1} 
\chi _j (x) = \theta   \Bigl( \frac { |x|} {a_j} \Bigr) .
\end{equation} 
satisfies the following properties.

\begin{enumerate}[{\rm (i)}] 

\item \label{ePES1:1} 
$  \| \chi _j \| _{ L^\infty  }\le 1 $ for all $j\ge 0$.

\item \label{ePES1:2b} 
$ |x|^{-2}\chi _j (x) \le a_j^{-2} \chi _j (x) $ for $x\in \R^N $ and all $j\ge 0$.

\item \label{ePES1:5} 
$e ^{t \Delta } \chi _j \le \chi  _{ j+1 } + 2^{\frac {N} {2}}e^{-\frac { a _{ j+1 }^2} {8t}}$ for all $j\ge 0$ and all $t\ge 0$.

\item \label{ePES1:4} 
$e ^{t \Delta } (  |\cdot |^{-2} \chi _j ) \le a_j^{-2} \chi  _{ j+1 } +a_j^{-2}  2^{\frac {N} {2}}e^{-\frac { a _{ j+1 }^2} {8t}} $ for all $j\ge 0$ and all $t\ge 0$.
\end{enumerate} 
\end{lem} 

\begin{proof} 
Property~\eqref{ePES1:1} is immediate, as well as Property~\eqref{ePES1:2b} (since $\chi _j (x)=0$ if $ |x|\le a_j$). Next, recall that, given $t, \nu >0$
\begin{equation} \label{fPES13} 
 (4\pi t)^{-\frac {N} {2}}\int  _{ \R^N  } e^{-\frac { |y|^2} {4 \nu  t}}dy= \nu^{\frac {N} {2}}  \pi ^{-\frac {N} {2}}\int  _{ \R^N  } e^{-  |y|^2 }dy =  \nu^{\frac {N} {2}} .
\end{equation} 
To prove~\eqref{ePES1:5}, note that 
\begin{equation} \label{fPES10} 
e ^{t \Delta } \chi _j (x)  =  (4\pi t)^{-\frac {N} {2}}\int  _{ \R^N  } e^{-\frac { |y|^2} {4t}} \chi _j (x-y)\, dy .
\end{equation} 
We claim that if $ |y| <a _{ j+1 } $, then 
\begin{equation} \label{fDFNt:1} 
\chi _j (x-y) \le \chi  _{ j+1 } (x)  \text{ for all } x\in \R^N .
\end{equation} 
Indeed, since $\theta $ is nondecreasing,
\begin{equation} \label{fDFNt:2} 
\chi _j ( x-y )=  \theta   \Bigl( \frac { |x- y|} {a_j} \Bigr) \le \theta   \Bigl( \frac { |x| + | y|} {a_j} \Bigr) \le \theta  \Bigl( \frac { |x| + a _{ j+1 }} {a_j} \Bigr) 
\end{equation} 
If $ |x|\le a _{ j+1 }$, then $ \frac { |x| + a _{ j+1 }} {a_j}\le 1 $, so that by~\eqref{fDFNt:2}, $\chi _j ( x-y )= 0$; and so~\eqref{fDFNt:1} holds. If $ |x|\ge a _{ j }$, then $\chi  _{ j+1 } ( x) =1$, so that~\eqref{fDFNt:1} holds. If $a _{ j+1 } \le  |x| \le a _{ j }$, then
\begin{equation*} 
\frac { |x| + a _{ j+1 }} {a_j} \le  \frac { 2 |x| } {a_j} = \frac { |x|} {a _{ j+1 }}
\end{equation*} 
so that by~\eqref{fDFNt:2} $\chi _j ( x-y ) \le \theta (\frac { |x|} {a _{ j+1 }}) = \chi  _{ j+1 } (x)$. This proves~\eqref{fDFNt:1}. 
We deduce from~\eqref{fDFNt:1} and~\eqref{fPES13}  with $\nu=1$ that
\begin{equation} \label{fPES11} 
 (4\pi t)^{-\frac {N} {2}}\int  _{ \{  |y|<a _{ j+1 } \}  } e^{-\frac { |y|^2} {4t}} \chi _j (x-y)\, dy \le \chi  _{ j+1 } (x).
\end{equation} 
Moreover, it follows from~\eqref{ePES1:1} and~\eqref{fPES13}  with $\nu=2$ that 
\begin{equation} \label{fPES12} 
\int  _{ \{  |y|>a _{ j+1 } \}  } e^{-\frac { |y|^2} {4t}} \chi _j (x-y) \, dy  \le e^{-\frac { a _{ j+1 }^2} {8t}} \int  _{\R^N  } e^{-\frac { |y|^2} {8t}}  dy = (8\pi t)^{\frac {N} {2}}e^{-\frac { a _{ j+1 }^2} {8t}}.
\end{equation} 
Property~\eqref{ePES1:5} follows from~\eqref{fPES10},  \eqref{fPES11} and~\eqref{fPES12},
then Property~\eqref{ePES1:4} follows from~\eqref{ePES1:2b} and~\eqref{ePES1:5}.
\end{proof} 

We next establish an estimate for the nonhomogeneous heat equation with a right-hand side given in terms of the functions $\chi _j$.

\begin{lem} \label{ePSS3} 
Let $m\ge 2$ be an integer, $\delta >0$, and let the sequence $(\chi _j) _{ j\ge 0 } $ be defined by~\eqref{fePSS1:1}. 
Set 
\begin{equation} \label{fPSDZ17b2} 
h( s, x)=  (   |x|^{-2} +1 )   \sum_{ j=1 }^{m } s^{\frac {j-1} {2}} \chi _j (x)   +  (  1 + s) s^{\frac {m-2 }{2}} 
\end{equation} 
for $s >0$ and $x\in \R^N $. It follows that $h \in  C([0, \infty ), L^\infty ( \R^N ) )$ and
\begin{equation} \label{fPSDZ20} 
\begin{split} 
  \int _0^t & e^{ (t-s) \Delta  } h(s) \, ds   \le   t^{\frac {1} {2}}  (1+  a _{ m }^{-2})  \Bigl( t^{\frac {m} {2}} + \sum_{ j=2 }^{m}   t^{\frac {j-1} {2}}  \chi   _j \Bigr) \\ & + m  2^{\frac {N} {2}}   (1+  a _{ m }^{-2})      (1+ t^{\frac {m+1} {2}}) e^{-\frac { a _{m+1} ^2} {8 t }}  +  2  \Bigl( \frac {1} {m} + \frac {t} {m+2} \Bigr) t^{\frac {m}{2}} 
\end{split} 
\end{equation} 
for all $t >0$.
\end{lem}

\begin{proof} 
The first property is immediate. 
Next, given $1\le j\le m$ and $0\le s\le t$, we deduce from Lemma~\ref{ePSS2}~\eqref{ePES1:5} and~\eqref{ePES1:4},  and the inequalities $a_m\le a_j$ and $a _{ m+1 }\le a_{j+1}$, that
\begin{equation*} 
 e^{ (t-s) \Delta  } [ ( |x|^{-2} + 1 ) \chi _j ]  \le (1+  a _{ m }^{-2}) \bigl(  \chi  _{ j+1 } + 2^{\frac {N} {2}} e^{-\frac { a  _{ m+1 } ^2} {8t }} \bigr) .
\end{equation*} 
It follows that
\begin{equation} \label{fPSDZ18} 
\begin{split} 
\int _0^t    s^{\frac {j-1} {2}}  e^{ (t-s) \Delta  }& [ ( |x|^{-2} + 1 ) \chi _j ]  \, ds \le  t^{\frac {j+1} {2}}  (1+  a _{ m }^{-2}) \bigl(  \chi  _{ j+1 } +  2^{\frac {N} {2}} e^{-\frac { a  _{ m+1 } ^2} {8t }} \bigr) \\
&\le    (1+  a _{ m }^{-2}) \bigl(  t^{\frac {j+1} {2}}\chi  _{ j+1 } +  (1+t^{\frac {m+1} {2}}) 2^{\frac {N} {2}} e^{-\frac { a  _{ m+1 } ^2} {8t }} \bigr) 
\end{split} 
\end{equation} 
 for $j\le m$.
Furthermore,
\begin{equation} \label{fPSDZ19} 
\int _0^t    e^{ (t-s) \Delta   } (1+ s) s^{\frac {m-2}{2}} ds = \int _0^t   ( 1+ s)   s^{\frac {m-2 }{2}} ds = 2 \Bigl( \frac {1} {m} + \frac {t} {m+2} \Bigr) t^{\frac {m}{2}}  .
\end{equation} 
Estimate~\eqref{fPSDZ20}  follows from~\eqref{fPSDZ17b2}, \eqref{fPSDZ18}, \eqref{fPSDZ19}, and the property $\chi _{m+1}\le 1$ 
\end{proof} 

We now estimate $\Md w$ for functions $w$ that are controlled in terms of the  $\chi _j$'s.

\begin{lem} \label{ePSS4} 
Let $\Omega $ be a domain of $\R^N $, possibly $\Omega =\R^N $, and 
let $\delta $, $\Psi $ and $U$ satisfy the hypotheses of Theorem~$\ref{ePSSZ1}$ (i.e., \eqref{ePSSZ1:1}, \eqref{ePSSZ1:4:b}, and~\eqref{ePSSZ1:1b1}--\eqref{ePSSZ1:4}).
Given $K>0$ and  $m \ge 2$, set 
\begin{equation} \label{fDTTabZ1} 
\Theta (t, x) =  \Bigl(  t^{\frac {m} {2}} + \sum_{ j=1 }^m t^{\frac {j-1} {2}} \chi _j   \Bigr) K
\end{equation} 
for $x\in \R^N $ and $t\ge 0$, where the sequence $(\chi _j) _{ j\ge 0 } $ is defined by~\eqref{fePSS1:1}, and let 
\begin{equation} \label{fTUQ} 
0 < T \le \frac {1} {4} .
\end{equation} 
 If $w\in L^\infty ((0,T) \times \Omega )$ and $ |w|\le \Theta $, then $\Md w$ defined by~\eqref{fPT3} belongs to $L^\infty ((0,T) \times \Omega )$. Moreover, 
\begin{equation}  \label{ePSS4:1} 
 | \Md w   | \le A (1 + K^{\alpha +1}) h 
\end{equation} 
and, if $z\in L^\infty ((0,T) \times \Omega )$ and $ |z|\le \Theta $, then
\begin{equation} \label{fpPSS4:21} 
  | \Md w - \Md z | \le  A ( 1 + K^{\alpha +1}  )    \Bigl|\frac {w  - z } {\Theta} \Bigr|  h
\end{equation} 
where $h$ is defined by~\eqref{fPSDZ17b2}, and the constant $A$ is independent of $T$, $m$, $K$, $w$ and $z$.
\end{lem}

\begin{rem} \label{ePSS4b1} 
As indicated just before the statement of  Theorem~$\ref{ePSSZ1}$, the function $\Theta$, here given by~\eqref{fDTTabZ1}, is used to define the function space
for the fixed point argument.  See also formula \eqref{fPSDZ14} below.  
As such, $\Theta$ has two key properties.
The first of these is that $\Theta (t,x)$ can balance any negative power of $t$ in the region $ |x|\le \delta 2^{-m}$ provided $m$ is chosen sufficiently large. 
Indeed, $\Theta (t,x) = K t^{\frac {m} {2}}$ for $t\ge 0$ and $ |x|\le \delta 2^{-m}$.
The second key property is that  $\Theta (0,x) \ge K$ if $ |x|\ge \delta $, so that if $\DIbd \in L^\infty  (\Omega ) $ vanishes in a neighborhood of $0$, then $\Theta (0,x) \ge w_0$ if $\delta $ is chosen sufficiently small and $K$ is chosen sufficiently large. 
Indeed, $\Theta (t, x)  = K\sum_{ j=0 }^m t^{\frac {j} {2}}\ge K$  for $t\ge 0$ and $ |x| \ge \delta $.
\end{rem} 

\begin{proof} [Proof of Lemma~$\ref{ePSS4}$]
It follows from~\eqref{fDTTabZ1}, \eqref{fePSS1:1}  and~\eqref{fTUQ} that if $t\le T$, then
\begin{equation} \label{fET1} 
 \| \Theta (t) \| _{ L^\infty (\R^N ) } \le K \sum_{ j=0 }^{m} t^{\frac {j} {2}} = K \frac {1- t^{\frac {m+1} {2}}} { 1 - t^{\frac {1} {2}}} \le 2K .
\end{equation} 
Moreover, we deduce from~\eqref{fPSDZ17b2} that
\begin{equation} \label{fET2} 
  \Theta (t)  \le Kh.
\end{equation} 
We note that by~\eqref{ePSSZ1:3}, \eqref{ePSSZ1:4} and the fact that $\nabla \Psi $ vanishes for $ |x| <\delta $ (by~\eqref{ePSSZ1:4:b})
\begin{equation}  \label{fpPSS4:1} 
 |2\nabla U \cdot \nabla \Psi + U\Delta \Psi| \le \CA1^{\frac {1} {\alpha }} (2 + \delta ^{-\frac {2} {\alpha }}) (  |\nabla \Psi |+  |\Delta \Psi |).
\end{equation} 
Next
\begin{equation}  \label{fpPSS4:2} 
\begin{split} 
\bigl| | \Psi U & +w|^\alpha (\Psi U  +w)   - \Psi  |U|^\alpha U \bigr|  \\ & \le   \bigl| | \Psi U+w|^\alpha (\Psi U+w) -  | \Psi U|^\alpha \Psi U \bigr|  +  (1- \Psi ^\alpha )  \Psi |U|^{\alpha +1} .
\end{split}  
\end{equation} 
We estimate the first term on the right-hand side of~\eqref{fpPSS4:2}  by using the elementary inequalities $  |\,  |z_1|^\alpha z_1 -  |z_2|^\alpha z_2| \le (\alpha +1) ( |z_1|^\alpha +  |z_2|^\alpha ) |z_1 -z_2|$ and $ |z_1 + z_2|^\alpha \le 2^\alpha (  |z_1|^\alpha +  |z_2|^\alpha )$; and we estimate the second term by using the fact that, by~\eqref{ePSSZ1:3}, $ |U|^{\alpha +1} $ is bounded on the support of $1- \Psi ^\alpha$, uniformly in $t$. 
We obtain
\begin{gather}  
  \bigl| | \Psi U+w|^\alpha (\Psi U+w) -  | \Psi U|^\alpha \Psi U \bigr|  \le 2^{\alpha +1} (\alpha +1) ( |U|^\alpha +  |w|^\alpha )  |w|  \label{fpPSS4:3} 
 \\   (1- \Psi ^\alpha )  \Psi |U|^{\alpha +1}  \le \CA1^{\frac {\alpha +1} {\alpha }} \delta ^{-\frac {2 (\alpha +1)} {\alpha }}  (1- \Psi ^\alpha )  \Psi . \label{fpPSS4:4} 
\end{gather} 
Therefore, by~\eqref{fPT3}, \eqref{fpPSS4:1},   \eqref{fpPSS4:2},  \eqref{fpPSS4:3} and~\eqref{fpPSS4:4}, there exists a constant $\CA2 $ independent of $t, x$ and $w$ such that
\begin{equation} \label{fpPSS4:5} 
 | \Md w| \le \CA2  ( (1- \Psi ^\alpha ) \Psi  +  |\nabla \Psi | +  |\Delta \Psi |) + \CA2  (  |U|^\alpha +  |w|^\alpha ) |w|  .
\end{equation} 
Note that $(1- \Psi ^\alpha ) \Psi  +  |\nabla \Psi | +  |\Delta \Psi | \le 3  \| \Psi \| _{ W^{2, \infty } }$ and vanishes on $\{  |x| \le \delta \}$. Therefore, we deduce from~\eqref{fePSS1:1} that
\begin{equation} \label{fpPSS4:6} 
(1- \Psi ^\alpha ) \Psi  +  |\nabla \Psi | +  |\Delta \Psi | \le  3  \| \Psi \| _{ W^{2, \infty } } \chi _0  \le  3  \| \Psi \| _{ W^{2, \infty } } \chi _1  \le  3  \| \Psi \| _{ W^{2, \infty } } h .
\end{equation} 
Moreover,  \eqref{fET1} and  \eqref{fET2} imply
\begin{equation} \label{fpPSS4:7} 
 |w|^{\alpha +1} \le  \| \Theta \| _{ L^\infty  }^\alpha \Theta \le (2K)^\alpha \Theta \le 2^\alpha K^{\alpha +1} h .
\end{equation} 
In addition, \eqref{ePSSZ1:3} implies
\begin{equation*} 
 \frac { |U|^{\alpha } } {K\CA1} \Theta  =  \frac { |U|^{\alpha } } {\CA1}  t^{\frac {m} {2}} +  \frac { |U|^{\alpha } } {\CA1}  \sum_{ j=1 }^m t^{\frac {j-1} {2}} \chi _j 
 \le  t^{-1} t^{\frac {m} {2}} +  |x|^{-2}  \sum_{ j=1 }^m t^{\frac {j-1} {2}} \chi _j  \le h
\end{equation*} 
so that
\begin{equation} \label{fpPSS4:8} 
 |U|^{\alpha }  |w| \le  |U|^{\alpha } \Theta \le  K \CA1 h.
\end{equation} 
Estimate~\eqref{ePSS4:1} follows from~\eqref{fpPSS4:5}, \eqref{fpPSS4:6},  \eqref{fpPSS4:7} and~\eqref{fpPSS4:8}. 

Next, given $w_1, w_2 \in L^\infty  ((0, T) \times \R^N )$ with $ |w_1|,  |w_2|\le \Theta$, we deduce from~\eqref{fPT3} that
\begin{equation*} 
  \Md w_1 - \Md w_2 =    | \Psi U+w_1 |^\alpha (\Psi U+w_1) -  | \Psi U+w_2 |^\alpha (\Psi U+w_2) 
\end{equation*} 
so that (for some constant $\CA3$)
\begin{equation} \label{fpPSS4:20} 
\begin{split} 
  | \Md w_1 - \Md w_2|  & \le  \CA3  (  |U|^\alpha +  |w_1|^\alpha +  |w_2|^\alpha ) |w_1 - w_2 | \\
  & =  \CA3  (  |U|^\alpha +  |w_1|^\alpha +  |w_2|^\alpha ) \Theta  \textstyle{ | \frac {w_1 - w_2} {\Theta} | } . 
\end{split} 
\end{equation} 
Since $ | w_j|^\alpha \Theta \le  \| \Theta \| _{ L^\infty  }^{\alpha } \Theta \le (2K)^{\alpha +1} h$ by~\eqref{fpPSS4:7}  and $|U|^{\alpha } \Theta \le  K \CA1 h$ by~\eqref{fpPSS4:8}, estimate~\eqref{fpPSS4:21} follows from~\eqref{fpPSS4:20} by possibly choosing $A$  larger still independent of $K$ (and using the inequality $K+K^{\alpha +1} \le 2(1+K^{\alpha +1}$).
\end{proof} 

We now can prove Theorem~\ref{ePSSZ1} by using a fixed point argument.

\begin{proof} [Proof of Theorem~$\ref{ePSSZ1}$]
We set 
\begin{equation} \label{fDK1} 
K= 2  \| \DIbd \| _{ L^\infty  }
\end{equation} 
and
\begin{equation} \label{fCTS} 
B= (1 + K^{\alpha +1})  A,
\end{equation} 
 where $A$ is given by Lemma~\ref{ePSS4}. 
 We now choose an integer $m\ge 2$ sufficiently large so that
\begin{equation} \label{fDFM} 
 \frac {4} {m}   \le \frac {K} {4B}.
\end{equation} 
Next, we fix $T$ satisfying~\eqref{fTUQ} sufficiently small so that
\begin{gather} 
2 ^{\frac {N} {2}} e^{-\frac {a_1^{2}}{8 t}} \le \frac {1} {2} t^{\frac {m}{2}} \text{ for all } 0<t\le T \label{fDFM:1}  \\
   T^{\frac {1} {2}}  (1+  a _{ m }^{-2})  \le \frac {K} {4B}  \label{fDFM:2} \\
  m   2^{\frac {N} {2}}   (1+  a _{ m }^{-2})      (1+ t^{\frac {m+1} {2}}) e^{-\frac { a _{m +1} ^2} {8 t }} \le  \frac {K} {4B} t^{\frac {m}{2}} \text{ for all } 0<t\le T .\label{fDFM:3} 
\end{gather} 
We let $\Theta $ be defined by~\eqref{fDTTabZ1} with $K$ given by~\eqref{fDK1} and we define the set $\Ens$ by
\begin{equation} \label{fPSDZ14} 
\Ens = \{ w\in L^\infty ((0,T) \times \Omega );\,    |w| \le \Theta  \} .
\end{equation} 
Given $w,z\in \Ens$ we set
\begin{equation*} 
\dist (w, z) =   \Bigl\| \frac {w-z} {\Theta} \Bigr\| _{ L^\infty ((0,T)\times \Omega  ) }
\end{equation*} 
so that $(\Ens, \dist )$ is a complete metric space.

By~\eqref{ePSSZ1:5}, we have  $ | \DIbd  |  \le  \| \DIbd  \| _{ L^\infty  } \chi _0 $, and we deduce from  Lemma~\ref{ePSS2}~\eqref{ePES1:5}, \eqref{fDK1}, and~\eqref{fDFM:1}   that
\begin{equation}  \label{fPSDZ15} 
 e ^{t \Delta _\Omega } | \DIbd  |   \le  \frac {K} {2} ( \chi _1 + 2^{\frac {N} {2}} e^{-\frac {a_1^{2}}{8 t}} )   \le \frac {K} {2}  \chi _1 + \frac {K} {4}   t^{\frac {m}{2}} 
  \end{equation} 
for all $0< t\le T$.  Moreover, it follows from~\eqref{ePSS4:1} and~\eqref{fCTS}  that 
\begin{equation} \label{fAVR1} 
 | \Md w   | \le B h \quad w\in \Ens
\end{equation} 
and from~\eqref{fpPSS4:21}  and~\eqref{fCTS} that  
\begin{equation} \label{fAVR2} 
  | \Md w - \Md z | \le  B   \dist (w, z) h \quad w, z\in \Ens
\end{equation} 
where $h$ is defined by~\eqref{fPSDZ17b2}. 
In particular, we see that $\Md w\in L^\infty ((0,T) \times \Omega )$ for all $w\in \Ens$.
We define $\Phi : \Ens \mapsto L^\infty ((0,T) \times \Omega )$ by
\begin{equation} \label{fNVDP1} 
 \Phi (w) (t) = e ^{t \Delta _\Omega } \DIbd  + \Map  ( \Md w) (t)
\end{equation} 
for $w\in \Ens$, where
\begin{equation*} 
\Map  (f) (t) = \int _0^t  e^{ (t-s) \Delta  _\Omega } f(s) \, ds 
\end{equation*} 
for $f\in L^\infty ((0,T)\times \Omega )$. 
It follows from~\eqref{fPSDZ20} and~\eqref{fTUQ} that
\begin{equation*} 
\begin{split} 
 | \Map (h) (t) | &   \le   T^{\frac {1} {2}}  (1+  a _{ m }^{-2})  \Bigl(   t^{\frac {m} {2}} +
  \sum_{ j=2 }^{m}   t^{\frac {j-1} {2}}  \chi   _j \Bigr) \\ & + m  2^{\frac {N} {2}}   (1+  a _{ m }^{-2})      (1+ t^{\frac {m+1} {2}}) e^{-\frac { a _{m+1} ^2} {8 t }}  +   \frac {4} {m}  t^{\frac {m}{2}}  .
\end{split} 
\end{equation*} 
Applying~\eqref{fDFM:2}, \eqref{fDFM:3}, and \eqref{fDFM}  we deduce that
\begin{equation} \label{fEMap1} 
 | \Map (h) (t) |   \le \frac {3K} {4B} t^{\frac {m} {2}}  +  \frac {K} {4B}
 \sum_{ j=2 }^{m}   t^{\frac {j-1} {2}}  \chi   _j .
\end{equation} 
It follows from~\eqref{fNVDP1}, \eqref{fPSDZ15}, \eqref{fAVR1}, and~\eqref{fEMap1} that
\begin{equation} \label{fFNZ4b1} 
 | \Phi (w) (t) |  \le  K  t^{\frac {m}{2}}  + \frac {K} {2}  \chi _1   + 
 \frac {K} {4}  \sum_{ j=2 }^{m}   t^{\frac {j-1} {2}}  \chi   _j  \le  \Theta 
\end{equation} 
where the last inequality follows from~\eqref{fDTTabZ1}.

Similarly, applying~\eqref{fNVDP1}, \eqref{fAVR2}, and~\eqref{fEMap1} we obtain 
\begin{equation} \label{fFNZ3b1} 
  | \Phi (w) - \Phi (z) |   \le \frac {3K} {4}   \Bigl( t^{\frac {m} {2}} + \sum_{ j=2 }^{m}   t^{\frac {j-1} {2}}  \chi   _j \Bigr) \dist (w,z)   \le  \frac {3} {4} \Theta \dist (w, z) .
\end{equation} 
Estimate~\eqref{fFNZ4b1}  implies that $\Phi : \Ens \to \Ens$,  then~\eqref{fFNZ3b1} implies that $\Phi $ is a strict contraction. 
Thus $\Phi $ has a unique fixed point $w \in  \Ens$, which proves property~\eqref{ePSSZ1:11}. Since $w\in \Ens$, we have $\Md w\in L^\infty ((0,T) \times \Omega )$ by~\eqref{fAVR1}, which proves property~\eqref{ePSSZ1:10}. Property~\eqref{ePSSZ1:12} follows, see Lemma~\ref{eHO2}.
\end{proof} 

\section{Perturbations of self-similar solutions} \label{sPSSS} 

Consider the equation~\eqref{Inlh} 
set on $\R^N $. Theorem~\ref{ePSSZ1} yields the following result.

\begin{thm} \label{ePSS1} 
Let $N\ge 1$ and $\alpha >0$. 
Suppose $U$ is a radially symmetric, regular self-similar solution of~\eqref{Inlh} on $\R^N $ with initial value $\MUU  |x|^{-\frac {2} {\alpha }}$, for some $\MUU \in \R$, in the sense~\eqref{fEP2}. 
Let $\DI \in L^\infty _\Loc (\R^N \setminus \{0\}) \cap L^\infty  ( \{  |x|>1 \} )  $ and suppose that there exists $\delta >0$ such that
\begin{equation} \label{fePSS1:1zz} 
\DI (x) = \MUU  |x|^{-\frac {2} {\alpha }}, \quad  |x|<\delta .
\end{equation} 
It follows that there exist $T>0$ and a solution $u\in C ((0, T], \Cbu )$ of~\eqref{Inlh} on $\R^N $ such that $u (t)  \to  \DI$ as $t\to 0$ in $L^p _\Loc (\R^N \setminus \{0\}) $ for all $1\le p<\infty $. 
Moreover, the following properties hold.
\begin{enumerate}[{\rm (i)}] 

\item \label{ePSS1:1} 
$u- U \in L^\infty ((0,T) \times \R^N )$ and there exists a constant $C$ such that
\begin{equation} \label{fePSS1:1:1} 
 \| u(t) - U(t) - e^{t \Delta } (\DI - \mu  |\cdot |^{-\frac {2} {\alpha }}) \|  _{ L^\infty }\le C t
\end{equation} 
for all $0<t<T$.

\item \label{ePSS1:2} 
If $\alpha >\frac {2} {N}$, then $u$ is a solution of the integral equation
\begin{equation} \label{NLHI} 
u(t) = e^{t\Delta } \DI + \int _0^t e^{(t-s) \Delta }  |u(s)|^\alpha u(s) \, ds
\end{equation} 
where the integral is convergent  in $L^r (\R^N ) + L^\infty  (\R^N ) $ for all $1\le  r < \frac {N\alpha } {2}$ and  each term is in $C((0,T ), \Cbu )$. Moreover, the map $t \mapsto u (t) - e^{t \Delta } \DI $ is in $ C([0,T ), L^r (\R^N ) + L^\infty  (\R^N ) )$ for all $1\le  r < \frac {N\alpha } {2}$.

\item \label{ePSS1:3} 
If 
\begin{equation} \label{ePSS1:1b} 
\esup  _{  | x | >R }  | \DI (x) | \goto  _{ R \to \infty  } 0
\end{equation} 
 then $u \in C( ( 0, T], \Cz )$. 
\end{enumerate} 
\end{thm} 

\begin{rem} \label{eRemNU4} 
In part~\eqref{ePSS1:2} of Theorem~\ref{ePSS1}, if in addition $\DI\in L^r (\R^N  ) $ for some $r \ge 1$ such that $\frac {N\alpha } {2(\alpha +1) } < r < \frac {N\alpha } {2}$, we expect that the solution $u $ will be in $C([0, T), L^r (\R^N ) )$, instead of $C([0, T), L^r (\R^N )+ L^\infty  (\R^N )  )$ as stated.
The obstacle to proving this is that the function $\Theta $ in the proof of Theorem~\ref{ePSSZ1} (see formula~\eqref{fDTTabZ1}) does not decay to $0$ at infinity.
\end{rem} 

\begin{rem} \label{eRemNU1} 
Suppose $U^1 \not = U^2$ are two radially symmetric, regular self-similar solutions of~\eqref{Inlh} on $\R^N $ with the same initial value $\MUU  |x|^{-\frac {2} {\alpha }}$, where $\MUU \in \R$, in the sense~\eqref{fEP2}. 
Suppose $u^1, u^2$ are solutions of~\eqref{Inlh}  on $(0,T)$, which are perturbations of the solutions $U^1, U^2$, respectively, in the sense of Theorem~\ref{ePSS1}, with the same initial value $\DI$. 
It follows that $u^1 \not = u^2$. More precisely,
\begin{equation} \label{eRemNU1:1} 
\liminf  _{ t\downarrow 0 } t^{\frac {1} {\alpha }}  | u^1 (t, 0)- u^2 (t, 0) | >0.
\end{equation} 
This is clear, since $U^1, U^2$ correspond to two profiles $f^1, f^2$ with $f^1 (0) \not = f^2 (0)$. (Otherwise, $f^1= f^2$, cf. equation~\eqref{fpr3}-\eqref{fpr3:1}.)
Therefore, $  t^{ \frac {1} {\alpha }}  |U^1(t,0) - U^2 (t, 0) | = | f^1 (0) - f^2 (0) | \not = 0$. Since $  | u^1(t, 0) - U^1 (t, 0) |$ is bounded by~\eqref{fePSS1:1:1}, the lower estimate~\eqref{eRemNU1:1} follows. 
\end{rem} 

\begin{proof} [Proof of Theorem~$\ref{ePSS1}$]
 Let 
\begin{equation*} 
\DIbd (x)= \DI (x) - \MUU  |x|^{-\frac {2} {\alpha }} =
\begin{cases} 
0 &  |x|< \delta \\ \DI (x) - \MUU  |x|^{-\frac {2} {\alpha }} &  |x| >\delta .
\end{cases} 
\end{equation*} 
Applying Theorem~\ref{ePSSZ1} with $\Omega =\R^N $ and $\Psi \equiv 1$, it follows that there exist $T>0$ and a function $w\in L^{\infty}((0,T)\times \R^{N})$ such that $\Md w \in L^{\infty}((0,T)\times \R^{N})$, which is a solution of~\eqref{fPT4}. Note that, since $\Psi \equiv 1$ on $\R^N $, we have 
\begin{equation} \label{fAD1} 
\Md w=  | u|^\alpha u -  |U|^\alpha U, 
\end{equation} 
where
\begin{equation} \label{fAD2} 
u = U + w .
\end{equation} 
We claim that $u$ is a classical solution of~\eqref{Inlh}  on $(0,T)\times \R^N $.
To see this, we first observe, as shown in Appendix~\ref{sLHERN}, that equation~\eqref{fPT4} implies that $w\in C((0, T] , \Cbu )$, and that, given any $0<\tau <T$,
\begin{equation}  \label{fPT4:b1} 
w (t+ \tau ) = e^{t \Delta  } w(\tau ) + \int _0^t e^{(t -s) \Delta  } \Md w (\tau +s) \, ds
\end{equation}  
for $0\le t\le T - \tau $. Since $U$ is a classical solution of~\eqref{Inlh}  on $(0,\infty )\times \R^N $, we have
\begin{equation}  \label{fPT4:b2} 
U (t+ \tau ) = e^{t \Delta  } U (\tau ) + \int _0^t e^{(t -s) \Delta  }  |U (\tau +s)|^\alpha U (\tau +s) \, ds
\end{equation}  
for all $t\ge 0$. Since $ |w|^\alpha w$ and $ |U|^\alpha U$ both belong to $C([\tau ,T], \Cbu )$, we deduce from~\eqref{fAD1}, \eqref{fAD2}, \eqref{fPT4:b1} and~\eqref{fPT4:b2} that $u\in C((0, T], \Cbu )$ satisfies
\begin{equation}  \label{fPT4:b3} 
u (t+ \tau ) = e^{t \Delta  } u (\tau ) + \int _0^t e^{(t -s) \Delta  }  |u (\tau +s)|^\alpha u (\tau +s) \, ds
\end{equation}  
for all $0\le t\le T-\tau $. Applying Theorem~\ref{eRHE0}, we conclude that $u$ is indeed a classical solution of~\eqref{Inlh}  on $(0,T)\times \R^N $.
Moreover, $ \| w(t) - e^{t \Delta } \DIbd \| _{ L^\infty  } \to 0$ as $t\downarrow 0$ by Theorem~\ref{ePSSZ1}~\eqref{ePSSZ1:12}, and it follows (see Lemma~\ref{eHO1}) that $w(t) \to \DIbd$ in $L^p _\Loc (\R^N )$ for all $1 \le p <\infty $. Applying Proposition~\ref{eRESS1}, we conclude that $u (t)  \to  \DI$ in $L^p _\Loc ( \R^N  \setminus \{ 0\})$, which proves the first part of Theorem~~\ref{ePSS1}.

Property~\eqref{ePSS1:1} follows from  Theorem~\ref{ePSSZ1}~\eqref{ePSSZ1:12}, since $u(t) - U(t) - e^{t \Delta } (\DI - \mu  |\cdot |^{-\frac {2} {\alpha }})= w(t)- e^{t\Delta }\DIbd$. 

Next, we prove Property~\eqref{ePSS1:2}, so we assume $\alpha >\frac {2} {N}$. 
Since $\DIbd = \DI  - \MUU  |\cdot |^{-\frac {2} {\alpha }}$,  it follows from equation~\eqref{fPT4}, and~\eqref{fAD1}-\eqref{fAD2} that
\begin{equation} \label{fNZ2:b2} 
u (t) = U(t) +  e^{t \Delta  } \DI - \MUU  e^{t \Delta  }  |\cdot |^{-\frac {2} {\alpha }} + \int _0^t e^{(t -s) \Delta  }( |u|^\alpha u -  |U|^\alpha  U) (s) \, ds .
\end{equation} 
Note that $ |u|^\alpha u -  |U|^\alpha  U = \Md w \in L^{\infty}((0,T)\times \R^{N})$, so that the integral on the right-hand side of~\eqref{fNZ2:b2} is in $C([0, T], \Cbu )$ by Lemma~\ref{eHRN2}. 
We claim that 
\begin{equation} \label{fNZ2:b4} 
\int _0^t e^{(t -s) \Delta  }( |u|^\alpha u -  |U|^\alpha  U) (s)   = \int _0^t e^{(t -s) \Delta  } |u|^\alpha u (s)   - \int _0^t e^{(t -s) \Delta  } |U|^\alpha  U (s) 
\end{equation} 
where the integrals in the right-hand side of~\eqref{fNZ2:b4} are convergent in $L^r (\R^N ) +L^\infty  (\R^N ) $ for all $1\le r< \frac {N\alpha } {2}$. 
Indeed, by Proposition~\ref{eRESS1}, the second integral is convergent in $L^p(\R^N )$ for all $p \ge 1$,  $p >\frac {N\alpha } {2(\alpha +1) } $. 
Moreover, 
\begin{equation*} 
 \big| |u|^\alpha u \big| \le C (  |U|^{\alpha +1} +  |w|^{\alpha +1} ) .
\end{equation*} 
Since
\begin{equation}  \label{feQSol3} 
   \| U(s)\| _{ L^{ ( \alpha +1) r } }^{\alpha +1}
=   s^{- \frac {\alpha +1} {\alpha } + \frac {N} {2 r }}  \| f  \| _{ L^{(\alpha +1) r } }^{\alpha +1} 
\end{equation} 
for all $s>0$, $- \frac {\alpha +1} {\alpha } + \frac {N} {2 r } >-1$, and  $w\in L^\infty ((0,T) \times \R^N )$, we see that  the first integral is convergent in $L^p(\R^N ) + L^\infty  (\R^N ) $.
Now if $p $ is as above and $1\le r\le \frac {N\alpha } {2 (\alpha +1)}$, we have $L^p (\R^N ) \hookrightarrow L^r(\R^N ) + L^\infty  (\R^N )$, so that both integrals are convergent in $L^r(\R^N ) + L^\infty  (\R^N )$. This proves the claim~\eqref{fNZ2:b4}. 
Equation~\eqref{NLHI} follows from~\eqref{fNZ2:b2}, \eqref{fNZ2:b1} and~\eqref{fNZ2:b4}, hence Property~\eqref{ePSS1:2} is established.

We finally prove Property~\eqref{ePSS1:3}, so we assume~\eqref{ePSS1:1b}. We use a comparison argument. Let $\xi \in C^\infty _\Comp (\R^N )$ satisfy $0\le \xi \le 1$ and $\xi (x)=1$ for $ |x|\le 1$, and set
\begin{equation} \label{fBD1} 
D =  \| \Md w \| _{ L^\infty ((0,T) \times \R^N ) } .
\end{equation} 
It follows from~\eqref{fAD1}-\eqref{fAD2} that 
$(1-\xi) |\Md w| \le C (1-\xi)( |U|^\alpha +  |w|^\alpha ) |w|$. Since $U$ is bounded on the support of $1-\xi $, we deduce that
\begin{equation} \label{fBD2} 
(1-\xi)\Md w = (1-\xi) \rho w
\end{equation} 
for some $\rho \in L^\infty ( (0,T) \times \R^N )$, and we let
\begin{equation} \label{fBD2:b1} 
K=  \| \rho \| _{ L^\infty  }.
\end{equation} 
On the other hand, it follows from~\eqref{ePSS1:1b} that there exists $z_0 \in \Cz$ such that
\begin{equation} \label{fBD3} 
 | \DIbd | \le z_0 
\end{equation} 
a.e. We let $z\in C ([0, T], \Cz)$ be the solution of 
\begin{equation}  \label{fBD4} 
\begin{cases} 
z_t - \Delta z = D \xi + K  z \\
z (0) =  z_0
\end{cases} 
\end{equation} 
so that $z\ge 0$, and $z$ is smooth on $(0,T) \times \R^N $.
Moreover, it follows from~\eqref{fBD1} and~\eqref{fBD2:b1} that
\begin{equation} \label{fBD2:b2} 
D \xi + K  z \ge   |\Md w| \xi  + (1- \xi ) \rho z .
\end{equation} 
Note also that both $u$ and $U$ are classical solutions of~\eqref{Inlh} on $(0,T) \times \R^N $, so that $w_t - \Delta w= \Md w$ on  $(0,T) \times \R^N $. Setting $W= w-z$ and applying~\eqref{fBD4}, \eqref{fBD2:b2}, and~\eqref{fBD2},  we deduce that  
\begin{equation} \label{fBD2:b3} 
\begin{split} 
W _t -   \Delta W  -  (1- \xi ) \rho W  
& = - D \xi - K  z +  \xi \Md w  +  (1- \xi )[ \Md w-  \rho W ] \\
&\le -(  | \Md w| - \Md w ) \xi  +  (1- \xi )[ \Md w- \rho w ]  \\
&= -(  | \Md w| - \Md w ) \xi \le 0 .
\end{split} 
\end{equation} 
Let now $\varphi (x)= \exp (- \sqrt {1 +  |x|^2})$, so that $\Delta \varphi \le 2\varphi $. 
Multiplying~\eqref{fBD2:b3} by $\varphi W^+$, where $W^+= \max\{ W, 0\}$ is the positive part of $W$, and integrating by parts, we obtain
\begin{equation} \label{fBD2:b4} 
\begin{split} 
\frac {1} {2} \frac {d} {dt} \int  _{ \R^N  } \varphi  (W^+)^2 & \le \int  _{ \R^N  } (1- \xi ) \rho \varphi (W^+)^2 - \int  _{ \R^N  } 
\nabla W \cdot \nabla [\varphi W^+ ] \\ &
\le K \int  _{ \R^N  } \varphi  (W^+)^2 - \int  _{ \R^N  } 
 W^+ \nabla (W^+) \cdot \nabla \varphi    \\ &
= K \int  _{ \R^N  } \varphi  (W^+)^2  + \frac {1} {2} \int  _{ \R^N  } 
(W^+)^2 \Delta \varphi \\ & \le (K + 2) \int  _{ \R^N  } \varphi  (W^+)^2.
\end{split} 
\end{equation} 
Note that all the above calculations are justified by the exponential decay of $\varphi $. 
Moreover, it follows from~\eqref{fePSS1:1:1} that $ \|w(t)- e^{t\Delta } w_0 \| _{ L^\infty  }\to 0$ as $t\to 0$.
Since $w_0\in L^\infty  (\R^N ) $, we also have $e^{t\Delta } w_0  \to w_0$ in $L^1_\Loc  (\R^N ) $,
so that $ W(t) \to w_0 - z_0 $ in $L^1_\Loc (\R^N ) $. Since $W\in L^\infty ((0,T)\times \R^N )$, we deduce that
 \begin{equation*} 
  \int  _{ \R^N  } \varphi  (W^+)^2 \goto  _{ t\to 0 }   \int  _{ \R^N  } \varphi  [( w_0-z_0 )^+]^2 =0.
 \end{equation*} 
This, together with inequality~\eqref{fBD2:b4}, implies that $W^+ \equiv 0$, so that $w\le z$. 
A similar calculation with $ \widetilde{W} = -w-z$ shows that $w\ge -z$.
Thus we see that $ |w|\le z$. It follows in particular that $  | w- e^{t \Delta } \DIbd | \le z + e^{t \Delta } z_0 \in C( [0,T], \Cz )$.  Since $  w- e^{t \Delta } \DIbd  \in C([0, T], \Cbu )$, we obtain $  w- e^{t \Delta } \DIbd  \in C([0, T], \Cz )$, and
Property~\eqref{ePSS1:3} easily follows.
\end{proof} 

\begin{proof} [Proof of Theorem~$\ref{ePR1}$]
{To show the existence of multiple solutions, apply} Theorem~\ref{ePSS1} to each of the infinitely many radially symmetric regular self-similar solutions of~\eqref{Inlh} given by Proposition~\ref{eRESS1}. The corresponding solutions are distinct by Remark~\ref{eRemNU1}. {Setting $\MUU_0=0$ if $\alpha \le \frac {2} {N}$, and $\MUU_0= [ \alpha ^{\frac {1} {\alpha  } }   [e^{ \Delta }  |\cdot |^{-\frac {2} {\alpha } }] (0) ]^{-1}$ 
if $\alpha > \frac {2} {N}$,  the fact that there is no local nonnegative solutions if $u_0 \ge 0$ and $\MUU>\MUU_0$ follows from Corollary~\ref{eFR3}.}
\end{proof} 

\begin{rem} \label{eRR1} 
We can let $\MUU =0$ in Theorem~\ref{ePR1}.  In particular, if we let $\DI \in L^\infty  (\R^N ) $ which vanishes in a neighborhood of $0$, then we obtain infinitely many sign-changing solutions of~\eqref{Inlh} (which have a singularity as $t\to 0$). This  extends the nonuniqueness results of~\cite{HarauxW, Weissler6}, which correspond to $\DI \equiv 0$.
\end{rem} 

\section{Sign-changing solutions on domains} \label{sSCSD} 

Let $\Omega $ be a bounded, smooth domain of $\R^N $ and assume $0\in \Omega $.
We consider the equation~\eqref{NLHD} and we look for singular solutions that behave like perturbations of self-similar solutions.

\begin{thm} \label{ePSD1} 
Let $\Omega $ be a bounded, smooth domain of $\R^N $, $N\ge 1$, and let $\alpha >0$. 
Suppose $U$ is a radially symmetric, regular self-similar solution of~\eqref{Inlh} on $\R^N $ with initial value $\MUU  |x|^{-\frac {2} {\alpha }}$, for some $\MUU \in \R$, in the sense~\eqref{fEP2}.
Suppose that there exists $\delta >0$ such that $\{  |x|<\delta \} \subset \Omega $ and let  $\DI \in L^\infty  (\Omega \cap \{  |x|> \delta \}) $ such that 
\begin{equation} \label{ePSD1:1zzz} 
\DI (x) = \MUU  |x|^{-\frac {2} {\alpha }}, \quad  |x|<\delta .
\end{equation} 
It follows that there exist $T>0$ and a solution $u\in C ((0, T], \Czo )$ of~\eqref{NLHD} such that  $u (t)  \to  \DI$ as $t\to 0$ in $L^p ( \Omega \cap \{  |x| > \varepsilon  \} )$ for all $ \varepsilon >0$ and all $p<\infty $.
Moreover, the following properties hold.
\begin{enumerate}[{\rm (i)}] 
\item \label{ePSD1:1} 
$u- U \in L^\infty ((0,T) \times \Omega )$. 

\item \label{ePSD1:2} 
If $\alpha >\frac {2} {N}$, then $u$ is a solution of the integral equation~\eqref{NLHID}  where  the integral is convergent in $L^r (\Omega  ) )$ for all $1\le  r < \frac {N\alpha } {2}$, and each term is in $C((0,T  ), \Czo )$. Moreover, $ u (t) \to \DI $ as $t\to 0$ in $   L^r (\Omega )$ for all $1\le  r < \frac {N\alpha } {2}$.
\end{enumerate}
\end{thm} 

\begin{rem} \label{eRemNU2} 
Suppose $U^1 \not = U^2$ are two radially symmetric, regular self-similar solutions of~\eqref{Inlh} on $\R^N $ with the same initial value $\MUU  |x|^{-\frac {2} {\alpha }}$, where $\MUU \in \R$, in the sense~\eqref{fEP2}. 
Suppose $u^1, u^2$ are solutions of~\eqref{Inlh}  on $(0,T)$, which are perturbations of the solutions $U^1, U^2$, respectively, in the sense of Theorem~\ref{ePSD1}. 
It follows that $u^1 \not = u^2$. More precisely, estimate~\eqref{eRemNU1:1} holds. 
This follows from the argument of Remark~\ref{eRemNU1}.
\end{rem} 

\begin{proof} [Proof of Theorem~$\ref{ePSD1}$]
We let $\nu >0$ be sufficiently small so that $\{  |x|<\delta +\nu \} \subset \Omega $, we fix a function $\Psi \in C^\infty _\Comp (\Omega )$ such that $0\le \Psi \le 1$, $\Psi (x)= 1$ for $ |x|\le \delta $,  $\Psi (x)=0$ for $ |x|\ge \delta +\nu$, and we define $\DIbd \in L^\infty (\Omega ) $ by
\begin{equation} \label{fCP1} 
\DIbd (x)= \DI (x) - \MUU  |x|^{-\frac {2} {\alpha }} \Psi (x) .
\end{equation} 
We see in particular that
\begin{equation} \label{fCP1:1} 
\DIbd (x) =0\quad  \text{if}\quad  |x|\le \delta .
\end{equation} 
Applying Theorem~\ref{ePSSZ1}, it follows that there exist $T>0$ and and a function $w\in L^{\infty}((0,T)\times \Omega )$ such that $\Md w \in L^{\infty}((0,T)\times \Omega )$, which is a solution of~\eqref{fPT4} with 
\begin{equation} \label{fCP4}
\Md w= 2\nabla U\cdot \nabla \Psi + U\Delta \Psi +  |u|^\alpha u - \Psi  |U|^\alpha U
\end{equation} 
where
\begin{equation} \label{fCP3}
u = \Psi U + w .
\end{equation} 
Note that $\Md w \in  L^{q}((0,T)\times \Omega )= L^q ((0, T), L^q (\Omega ) )$, for every $q<\infty $, so it follows easily from equation~\eqref{fPT4} that $w\in C ((0, T], \Czo )$, $w - e^{t\Delta _\Omega } \DIbd \in C ( [0,T], \Czo )$ and that, given any $0<\tau <T$,
\begin{equation}  \label{fCP5} 
w (t+ \tau ) = e^{t \Delta  _\Omega } w(\tau ) + \int _0^t e^{(t -s) \Delta  _\Omega } \Md w (\tau +s) \, ds
\end{equation}  
for $0\le t\le T - \tau $. Moreover, $V= \Psi U$ is $C^1 $ in $t$, $C^2$ in $x$, vanishes on a compact subset of $\Omega $, and satisfies the equation 
\begin{equation*} 
V_t - \Delta V= - 2\nabla U\cdot \nabla \Psi -  U\Delta \Psi + \Psi  |U|^\alpha U = - \Md w +  |u|^\alpha u
\end{equation*} 
on $(0,\infty ) \times \Omega $. It follows that 
\begin{equation}  \label{fCP6} 
V (t+ \tau ) = e^{t \Delta  _\Omega } V(\tau ) + \int _0^t e^{(t -s) \Delta  _\Omega } (- \Md w +  |u|^\alpha u) (\tau +s) \, ds
\end{equation}  
for $t\ge 0$. Summing~\eqref{fCP5} and~\eqref{fCP6}, we deduce that $u\in C((0,T), \Czo)$ satisfies 
\begin{equation}  \label{fCP7} 
u (t+ \tau ) = e^{t \Delta  _\Omega } u (\tau ) + \int _0^t e^{(t -s) \Delta  _\Omega }( |u|^\alpha u) (\tau +s) \, ds
\end{equation}  
for $0\le t\le T - \tau $. By standard regularity (see e.g. Theorem~\ref{eRHE2}), $u$ is a classical solution of~\eqref{NLHD} on  $(0,T) \times \Omega $.
In addition, it follows from~\eqref{fCP1} and~\eqref{fCP3} that
\begin{equation}  \label{fCP8} 
u(t) - \DI =  \Psi ( U (t)  - \MUU  |\cdot |^{-\frac {2} {\alpha }})+ w (t) - e^{t \Delta _\Omega }\DIbd  + e^{t \Delta _\Omega }\DIbd - \DIbd.
\end{equation} 
We have $ \| w(t) - e^{t \Delta _\Omega } \DIbd \| _{ L^\infty  } \to 0$ as $t\downarrow 0$ by Theorem~\ref{ePSSZ1}~\eqref{ePSSZ1:12}. Moreover, $\DIbd \in L^\infty  (\Omega ) $, so that $ \|  e^{t \Delta _\Omega } \DIbd - \DIbd \| _{ L^p  } \to 0$ as $t\downarrow 0$, for all $p<\infty $. Also, $U (t)  - \MUU  |\cdot |^{-\frac {2} {\alpha }} \to 0$ as $t\downarrow 0$ uniformly on $ \{  |x|>\varepsilon  \} $ for every $\varepsilon >0$, and this  proves the first part of the statement.

Next, we observe that by~\eqref{fCP3}, $u-U= (\Psi -1) U + w$. Since $1-\Psi $ vanishes in a neighborhood of $0$, it follows that $ (\Psi -1) U\in L^\infty ((0,\infty ) \times \R^N ))$. 
Moreover, $w\in L^{\infty}((0,T)\times \R^{N})$, and Property~\eqref{ePSD1:1} follows.

We now prove Property~\eqref{ePSD1:2}, so we suppose $\alpha >\frac {2} {N}$.
It follows in particular (see Proposition~\ref{eRESS1}) that  $U (t)  \to \MUU  |\cdot |^{-\frac {2} {\alpha }}$ in $L^p (\Omega ) $ as $t\downarrow 0$, for all $1\le p<\frac {N\alpha } {2}$. 
Therefore, we deduce from~\eqref{fCP8} that $u(t) \to \DI$  likewise.  
Moreover, $u\in C ((0, T], \Czo )$ is a solution of~\eqref{NLHD}, so that
\begin{equation}  \label{fCP9} 
u(t)= e^{ (t-\varepsilon )\Delta _\Omega } u(\varepsilon ) + \int _\varepsilon ^t e^{ (t-\varepsilon )\Delta _\Omega }  |u|^\alpha u(s)\,ds
\end{equation} 
for all $0<\varepsilon <t<T$.
Since $u(t) \to \DI$ in $L^1 (\Omega ) $, we see that
\begin{equation} \label{fCP10} 
e^{ (t-\varepsilon )\Delta _\Omega } u(\varepsilon ) \goto  _{ \varepsilon \downarrow 0 }  e^{ t\Delta _\Omega }  \DI 
\end{equation} 
in $L^\infty  ( \Omega ) $. 
Let now $1\le r< \frac {N\alpha } {2}$ and let $ \frac {N\alpha } {2(\alpha +1)}\le p< \frac {N\alpha } {2}$ be such that $p\ge r$. We have $ | u |^{\alpha +1}\le C (  |U|^{\alpha +1} +  |w|^{\alpha +1})$, and $ |w|^{\alpha +1}$ is bounded in $L^\infty  (\Omega ) $, hence in $L^r (\Omega ) $. Moreover, $  \|  |U|^{\alpha +1} \| _{ L^p } $, hence  $  \|  |U|^{\alpha +1} \| _{ L^r} $, is integrable on $(0,T)$, see formula~\eqref{feQSol3}. 
Therefore, one easily passes to the limit in~\eqref{fCP9} as $\varepsilon \downarrow 0$ and obtain equation~\eqref{NLHID}, where the integral is convergent in $L^r (\Omega ) $ for all $1\le r<  \frac {N\alpha } {2}$. Since the first two terms in~\eqref{NLHID} are in $C((0,\infty ), \Czo )$, so is the integral term.
This proves Property~\eqref{ePSD1:2}.
\end{proof} 

\begin{proof} [Proof of Theorem~$\ref{ePR2}$]
Without loss of generality, we suppose $x_0=0$. 
{To show the existence of multiple solutions, apply} Theorem~\ref{ePSD1} to each of the infinitely many radially symmetric regular self-similar solutions of~\eqref{Inlh} given by Proposition~\ref{eRESS1}, with $\DI (x) = \MUU  \zeta  (x)  |x|^{-\frac {2} {\alpha }}$. The corresponding solutions are distinct by Remark~\ref{eRemNU2}. {Setting $\MUU_0=0$ if $\alpha \le \frac {2} {N}$, and $\MUU_0= [ \alpha ^{\frac {1} {\alpha  } }   [e^{ \Delta }  |\cdot |^{-\frac {2} {\alpha } }] (0) ]^{-1}$ 
if $\alpha > \frac {2} {N}$,  the fact that there is no local nonnegative solutions if $u_0>0$ and $\MUU>\MUU_0$ follows from Corollary~\ref{eFR3}. }
\end{proof} 

\begin{rem} 
If we apply Theorem~\ref{ePR2} in the case $\mu =0$, we obtain a nonuniqueness result for equation~\eqref{NLHD} with initial values $\DI \in L^\infty  (\Omega ) $. The only requirement is that $\DI$ vanish on some open subset of $\Omega $. Unlike the results in~\cite{Baras} and~\cite[Theorem~15.3~(ii)]{QuittnerS}, we do not require radial symmetry or positivity.
\end{rem} 

\appendix

\section{The heat equation on a domain} \label{sLHEOM} 

We consider an open, connected subset $\Omega \subset \R^N $. 
We recall that the heat semigroup on $\Omega $ with Dirichlet boundary conditions, $(e^{t \Delta }) _{ t\ge 0 }$ is the strongly continuous semigroup on $L^2 (\Omega ) $ generated by the operator $\Delta $ with domain $\{u\in H^1_0 (\Omega ) ;\, \Delta u\in L^2 (\Omega )  \}$. 
We recall that $e^{t \Delta }$ is a contraction of $L^p (\Omega ) $ for every $1\le p\le \infty $. See e.g.~\cite[Theorem~1.3.3, p.~14]{Davies}.
The corresponding heat kernel  $G_\Omega  $ satisfies
\begin{gather} 
G_\Omega \in C^\infty ((0,\infty ) \times \Omega \times \Omega  ) \label{fHK1} \\
0 \le G_\Omega (t, x, y) \le K t^{-\frac {N} {2}} e^{- \frac { |x-y|^2} {\delta t}} \label{fHK2}
\end{gather} 
where $K, \delta >0$ are two constants independent of $t,x$. See e.g.~\cite{Davies}, in particular Theorem~5.2.1 p.~149 and Corollary~3.2.8 p.~89. 

\begin{lem} \label{eHO1} 
Let $\DI \in L^\infty  (\Omega  ) $ and let $u(t)= e^{t \Delta }\DI $ for $t\ge 0$, in the sense that
\begin{equation*} 
u(t, x)= \int _\Omega G_\Omega (t, x, y) \DI (y) \, dy
\end{equation*} 
It follows that $u \in  C((0, \infty )\times \Omega  )$. 
Moreover, if $1\le p<\infty $, then 
\begin{equation} \label{eHO1:1} 
  \|  u(t) - \DI \| _{ L^p (B) } \goto _{  t\downarrow 0 } 0 
\end{equation}  
for every bounded subset $B\subset \Omega $.
\end{lem} 

\begin{proof} 
The property $u \in  C((0, \infty )\times \Omega  )$ follows easily from the continuity property~\eqref{fHK1}, the bound~\eqref{fHK2}, and the dominated convergence theorem.
Let $R>0$ be such that $B\subset  \{  |x|<R \} $ and let
\begin{equation*} 
u_1 = \DI 1 _{ \{  |x|<R+1 \} } \quad u_2 = \DI - u_1. 
\end{equation*} 
In particular, $u_1$ has compact support, hence $u_1 \in L^p (\Omega ) $ for all $1\le p\le \infty $.  Since $\DI= u_1 + u_2$, we see that
\begin{equation} 
\begin{split} 
 \|  u(t) - \DI \| _{ L^p (B) } &  \le   \|   e^{t\Delta }u_1 -u_1  \| _{ L^p (B ) } +  \|   e^{t\Delta }u_2 - u_2 \| _{ L^p (B) } \\ &  \le   \|   e^{t\Delta }u_1 -u_1  \| _{ L^p (\Omega ) } +  \|   e^{t\Delta }u_2  \| _{ L^p (B) }
\end{split} 
\end{equation} 
since $u_2 =0$ on $B$. Since $u_1\in L^p (\Omega ) $ and $(e^{t \Delta }) _{ t\ge 0 }$ is a strongly continuous semigroup on $L^p (\Omega ) $, it follows that $  \|   e^{t\Delta }u_1 -u_1  \| _{ L^p (\Omega ) } \to 0$  as $t\to 0$. Thus we need only show that $  \|   e^{t\Delta }u_2  \| _{ L^\infty  (B) } \to 0$. This is immediate since by~\eqref{fHK2}, we have for every $x\in B$
\begin{equation*} 
 |e^{t\Delta }u_2 (x) | \le K  \| \DI \| _{ L^\infty  } t^{-\frac {N} {2}}  \int  _{  \Omega \cap \{  |y|>R+1 \} } e^{- \frac { |x-y|^2} {\delta t}} dy.
\end{equation*} 
Since $ |x|<R$ and $ |y|> R+1$ we have $ |x-y |\ge 1$ so that
\begin{equation*} 
 |e^{t\Delta }u_2 (x) | \le K  \| \DI \| _{ L^\infty  } e^{- \frac {1} {2\delta t}} t^{-\frac {N} {2}}  \int  _{ \R^N } e^{- \frac { |x-y|^2} { 2 \delta t}} = K  \| \DI \| _{ L^\infty  } (2\pi \delta )^{\frac {N} {2}} e^{- \frac {1} {2\delta t}} .
\end{equation*} 
It follows that $  \|   e^{t\Delta }u_2  \| _{ L^\infty  (B) } \to 0$, which completes the proof.
\end{proof} 

\begin{lem} \label{eHO2} 
Let $T>0$,  $f\in L^\infty ( (0,T) \times \Omega  ) $ and set
\begin{equation} \label{fHO1} 
\Phi (t,x)= \int _0^t \int  _{ \Omega   } G_\Omega (t-s, x, y) f(s, y) \, dy \, ds
\end{equation} 
for all $0<t<T$ and $x\in \Omega  $. It follows that $\Phi (t,x)$ is well defined for all $0<t<T$ and $x\in \Omega  $ as a Lebesgue integral on $(0,t)\times \Omega $, and that $\Phi \in C((0,T)\times \Omega )$. 
In addition, there exists a constant $C$ such that $  \| \Phi  (t, \cdot )\| _{ L^\infty  } \le C t $   for all $0 < t <T$. 
\end{lem} 

\begin{proof} 
Let $0<t<T$ and $x\in \Omega  $. It follows from~\eqref{fHK1} that $G_\Omega (t-s, x, y) f(s, y)$ is a measurable function of $(s,y)\in (0,t) \times \Omega $. Moreover, the Gaussian bound~\eqref{fHK2} implies that
\begin{equation}  \label{fHO2} 
 | G_\Omega (t-s, x, y) f(s, y) | \le  K  \| f \| _{ L^\infty  }  (t-s) ^{-\frac {N} {2}} e^{- \frac { |x-y|^2} {\delta (t-s)}}. 
\end{equation} 
Since 
\begin{equation} \label{fHO3} 
 (t-s) ^{-\frac {N} {2}}  \int  _{ \R^N  } e^{- \frac { |x-y|^2} {\delta (t-s)}} dy = (\pi \delta  )^{\frac {N} {2}}
\end{equation} 
the right-hand side of~\eqref{fHO2}  is clearly integrable on $(0,t) \times \Omega $. Thus the integral~\eqref{fHO1} is well defined.  
To show the continuity of $\Phi $, fix $(t,x)\in (0,T)\times \Omega $ and let $(t_n ) _{ n\ge 1 }\subset (0, T) $ and $(x_n) _{ n\ge 1 }\subset \Omega $ satisfy $t_n \to t$ and $x_n \to x$ as $n\to \infty $. 
Fix $0 < h < \frac {t} {2}$, so that there exists $\tau >0$ such that
\begin{equation} \label{fHO3:1} 
\tau +h \le t_n, t \le \tau + 2h .
\end{equation} 
We have
\begin{equation}\label{fHO7} 
\begin{split} 
\Phi (t,x)-\Phi (t_n , x_n )&=\int _0^{ \tau  }  \int  _{ \Omega   } (G_\Omega (t-s, x, y)-G_\Omega (t_n -s,  x_ n, y)) f(s, y) \, dy \, ds\\
&+\int _{ \tau  } ^t  \int  _{ \Omega   } G_\Omega (t-s, x, y)\, dy \, ds\\
&- \int _\tau^{t_n } \int  _{ \Omega   } G_\Omega (t_n -s, x_n , y) f(s, y) \, dy \, ds \\ & \Eqdef I_1+I_2+I_3.
\end{split} 
\end{equation}
Applying~\eqref{fHO2} and~\eqref{fHO3}, we see that
\begin{equation} \label{fHO3:2} 
|I_2| \le K\| f \| _{ L^\infty  }\int _{ \tau } ^t \int  _{ \Omega   } (t-s) ^{-\frac {N} {2}} e^{- \frac { |x-y|^2} {\delta (t-s)}} \, dy \, ds\\
 \le 2K\| f \| _{ L^\infty  }(\pi \delta  )^{\frac {N} {2}} h .
\end{equation} 
Similarly,
\begin{equation} \label{fHO3:3} 
|I_3| \le 2K\| f \| _{ L^\infty  }(\pi \delta  )^{\frac {N} {2}} h .
\end{equation} 
Next, we show that $I_1 \to 0$ as $n\to \infty $ by dominated convergence. 
Indeed, the integrand converges pointwise to $0$ by~\eqref{fHK1}. 
Therefore, it suffices to show that the integrand is bounded by a fixed function in $L^1 ( (0,T) \times \Omega  )$. Indeed, by~\eqref{fHK2} 
\begin{equation*} 
G_\Omega (t-s, x, y) \le K (t - s) ^{-\frac {N} {2}} e^{- \frac { |x-y|^2} {\delta (t - s ) }} .
\end{equation*} 
Since $ h \le t-s \le T$, it follows that
\begin{equation*} 
G_\Omega (t-s, x, y) \le K h ^{-\frac {N} {2}} e^{- \frac { |x-y|^2} {\delta T }} 
\end{equation*} 
and, similarly,
\begin{equation*} 
G_\Omega (t_n -s, x_n , y) \le K h ^{-\frac {N} {2}} e^{- \frac { |x_n -y|^2} {\delta T }} 
\end{equation*} 
For $n$ large, $ |x_n - y|^2 \ge \frac {1} {2}  |x-y |^2- 1$ for all $y\in \R^N $, so that 
\begin{equation*} 
 |G_\Omega (t-s, x, y)-G_\Omega (t_n -s,  x_ n, y))| \le  C h ^{-\frac {N} {2}} e^{- \frac { |x-y|^2} {2 \delta T }} 
\end{equation*} 
which shows that $I_1 \to 0$ as $n\to \infty $.
Together with~\eqref{fHO3:2} and~\eqref{fHO3:3},  this implies that 
\begin{equation*} 
\limsup  _{ n\to \infty  }  | \Phi (t,x)-\Phi (t_n , x_n ) | \le 4 K\| f \| _{ L^\infty  }(\pi \delta  )^{\frac {N} {2}} h .
\end{equation*} 
The conclusion follows by letting $h\to0$.
\end{proof} 

\section{The heat equation on $\R^N $} \label{sLHERN} 

We let
\begin{equation} 
K(t, x) =  (4\pi t)^{- \frac {N} {2}} e^{- \frac { |x|^2} {4t}}, \quad t>0, x\in \R^N 
\end{equation} 
so that $G (t, x,y) = K (t, x-y)$. 

\begin{lem} \label{eHRN1} 
Let $\DI \in L^\infty  (\R^N ) $ and set $u(t)= e^{t \Delta }\DI $ for $t\ge 0$. 
It follows that $u \in  C((0, \infty ), \Cbu )$. 
\end{lem} 

\begin{proof} 
Note that $u (t) = K(t, \cdot ) \star \DI$. Since $K\in C((0,\infty ), W^{1, 1} (\R^N ) )$, we see that $u\in  C((0,\infty ), W^{1, \infty } (\R^N ) )$. Hence the result, since $W^{1, \infty } (\R^N ) \hookrightarrow \Cbu $.
\end{proof} 

\begin{lem} \label{eHRN1:b1} 
The heat semigroup $(e^{t\Delta }) _{ t\ge 0 }$ is a $C_0$ semigroup of contractions on $\Cbu$. 
In addition, $(e^{t\Delta }) _{ t\ge 0 }$ is an analytic semigroup on $\Cbu$.
Moreover, both statements are true if $\Cbu$ is replaced by $\Cz$.
\end{lem} 

\begin{proof} 
We first prove that $(e^{t\Delta }) _{ t\ge 0 }$ is a $C_0$ semigroup of contractions on $\Cbu$. 
Let $\DI \in \Cbu$ and set $u(t)= e^{t\Delta } \DI$. By Lemma~\ref{eHRN1} it suffices to show that $  \| u(t) - \DI \| _{ L^\infty  } \to 0$ as $t\to 0$. Note that 
\begin{equation*} 
u(t,x) - \DI (x)   = \int  _{ \R^N  } K(t, x-y ) \DI (y) \, dy - \DI (x) \int  _{ \R^N  } K(t, x-y )  \, dy 
\end{equation*} 
so that
\begin{equation} 
\begin{split} 
 |u (t, x) - \DI (x)| \le &  \int  _{ \{  |x-y|<\delta  \} } K(t, x-y ) |\DI (y) - \DI (x)| \, dy \\
& +  \int  _{ \{  |x-y|>\delta  \} } K(t, x-y ) |\DI (y) - \DI (x)| \, dy \\ & = I_1 + I_2 .
\end{split} 
\end{equation} 
Let $\varepsilon >0$. There exists $\delta >0$ such that $ | \DI (x)- \DI (y) |\le \frac {\varepsilon } {2}$ for $ |x-y|\le \delta $. Therefore,
\begin{equation} 
I_1 \le \frac {\varepsilon } {2} \int  _{ \R^N  } K(t, y )\, dy = \frac {\varepsilon } {2} .
\end{equation} 
Next, 
\begin{equation} 
I_2 \le 2  \| \DI \| _{ L^\infty  } \int  _{ \{  |y|>\delta  \} } K(t, y ) \, dy \goto _{ t\to 0 }0
\end{equation} 
Hence the result. 

To prove analyticity on $\Cbu$, we observe that by~\cite[Chapter~IX, Section~10]{Yosida}, it suffices to show that 
\begin{equation*} 
\sup  _{ 0<t\le 1 }  \| t \Delta e^{t \Delta } \| _{ \mathcal L (\Cbu) } <\infty .
\end{equation*} 
Since $ e^{t \Delta } \DI = K(t, \cdot ) \star \DI$, we see that $t \Delta e^{t \Delta } \DI = (t \Delta  K(t, \cdot ) )\star \DI$. Therefore, it suffices to prove that
\begin{equation*} 
\sup  _{ 0<t\le 1 }  \| t \Delta  K(t, \cdot )  \| _{ L^1} <\infty 
\end{equation*} 
which follows from elementary calculations. 

Finally, since $\Cz $ is a closed subspace of $\Cbu$, which is invariant under the action of $e^{t \Delta }$, 
the corresponding statements for $\Cz$ are an immediate consequence.
\end{proof} 

\begin{lem} \label{eHRN2}
Let $T>0$,  $f\in L^\infty ( (0,T) \times \R^N ) $ and set
\begin{equation} 
\Phi (t,x)= \int _0^t \int  _{ \R^N  } K(t-s, x-y ) f(s, y) \, dy \, ds
\end{equation} 
for all $0<t<T$ and $x\in \R^N $. It follows that $\Phi \in C([0, T], \Cbu )$. 
Moreover, given any $\tau >0$, 
\begin{equation} \label{eHRN2:1}
e^{\tau \Delta } \Phi (t, x)= \int _0^t \int  _{ \R^N  } K(\tau + t-s, x-y ) f(s, y) \, dy \, ds
\end{equation} 

\end{lem} 

\begin{proof} 
We define 
\begin{equation*} 
 \widetilde{K} (t,x) = 
 \begin{cases} 
 K(t, x) & x\in \R^N , 0< t <T \\ 0 & x\in \R^N ,  t(T-t) <0
 \end{cases}  
\end{equation*} 
so that $ \widetilde{K} \in L^1 (\R \times \R^N ) $ and
\begin{equation*} 
 \widetilde{f} (t,x) =
 \begin{cases} 
 f(t, x) & x\in \R^N , 0<t<T \\  0 & x\in \R^N , t(T-t) <0
 \end{cases}  
\end{equation*} 
so that $ \widetilde{f} \in   L^\infty  (\R \times \R^N )$. We may write
\begin{equation*} 
\Phi (t,x)= \int  _{ \R\times \R^N  }  \widetilde{K } (t-s, x-y)  \widetilde{f} (s,y) \,dy\,ds 
\end{equation*} 
which means that $\Phi =  \widetilde{K}  \widetilde{\star }  \widetilde{f}  $, where $ \widetilde{\star} $ is the convolution on $\R \times \R^N $. Since $ \widetilde{K} \in L^1$ and $ \widetilde{f} \in L^\infty $, we have $\widetilde{K}  \widetilde{\star }  \widetilde{f}   \in C _{ \mathrm {b,u} } (\R \times \R^N ) \subset C(\R, \Cbu )$. This proves the first part of the result.
Identity~\eqref{eHRN2:1} easily follows from the above considerations and the fact that if $\tau >0$ and $0<s<t$, then $K (\tau , \cdot ) \star K( t-s, \cdot )= K( \tau +t-s, \cdot )$.
\end{proof} 

\section{Regularity for the nonlinear heat equation} \label{sRHE} 

We begin with the case of the heat equation set on $\R^N $.

\begin{thm} \label{eRHE0} 
Let $X$  be either $ \Cz$ or $\Cbu$. 
Let $\alpha >0$, $\DI \in  X$, $T>0$, and suppose $u\in C( [0,T], X)$ satisfies
\begin{equation}  \label{eRHE0:1} 
u(t)= e^{t \Delta } \DI +  \int _0^t e^{ (t -s) \Delta }  |u(s)|^\alpha u(s)  \, ds
\end{equation} 
for all $0\le t\le T$. It follows that $u \in C^1 ( (0,T), X)$, $\Delta u \in C ( (0,T), X)$, and $u_t= \Delta u+  |u|^\alpha u$ for all $0<t<T$.
In addition, $\nabla u \in C^1 ( (0,T ), X)$, $\Delta \nabla u\in C ( (0,T ), X)$, and all space derivatives of $u$ of order  two are in $C((0,T), X)$.
\end{thm} 

We use the following lemma.

\begin{lem} \label{eRHE1} 
Let $e^{tA}$ be a $C_0 $ semigroup of contractions on a Banach space $X$, which is also an analytic semigroup. 
Suppose $f\in C([0,T], X) $ and let 
\begin{equation*} 
v(t)= \int _0^t e^{ (t -s) A } f(s) \, ds
\end{equation*} 
for all $0\le t\le T$.
Given any $0<\gamma <1$, the map $v:[0 ,T] \to X$ is H\"older continuous of exponent $\gamma $.
\end{lem} 

Lemma~\ref{eRHE1} is proved in a slightly different context in~\cite[Proposition~1.2]{Weissler1}. For completeness, we give the proof here.
We use the following estimates which involve the $\Gamma $ function 
\begin{equation*} 
\Gamma (\gamma )= \int _0^\infty  s^{\gamma -1} e^{- s} \, ds.
\end{equation*} 
(See the proof of~\cite[Theorem~11.3]{Komatsu} for~\eqref{fKTS1},  and~\cite[Theorem~12.1]{Komatsu} for~\eqref{fKTS3})

\begin{lem} \label{eKTS1} 
Let $(e^{tA}) _{ t\ge 0 }$ be a $C_0 $ semigroup of contractions on a Banach space $X$. 
Let $\lambda >0$, $ \gamma \in (0,1)$ and
\begin{equation}  \label{fKTS2} 
(\lambda - A)^{- \gamma } f = \frac {1} { \Gamma (\gamma )} \int _0^\infty  s^{\gamma -1} e^{- \lambda s} e^{ sA} f \, ds
\end{equation} 
for $f\in X$. It follows that 
\begin{equation} \label{fKTS1} 
 \| [ e^{t A}  - I ] (\lambda - A)^{- \gamma }   \| _{ \mathcal L  ( X) } \le \frac {2} {\Gamma (\gamma +1)}  t^\gamma 
\end{equation} 
for all $t\ge 0$. In addition, if $(e^{tA}) _{ t\ge 0 }$ is an analytic semigroup, then there exists a constant $C$ such that
\begin{equation} \label{fKTS3} 
 \| (\lambda - A)^{ \gamma } e^{-t(\lambda -A)} \| _{ \mathcal L  ( X) } \le C t^{-\gamma } 
\end{equation} 
for all $t> 0$.
\end{lem} 

\begin{proof} 
Given $f\in X$,
\begin{equation*} 
\begin{split} 
\Gamma (\gamma ) [e^{t A} &- I ]  (\lambda - A)^{- \gamma }  f   \\  & =
 \int _0^\infty  s^{\gamma -1}  e^{- \lambda s} e^{ (t+ s)A} f \, ds
-   \int _0^\infty  s^{\gamma -1} e^{- \lambda s} e^{ sA} f \, ds \\ &
= 
  \int _t ^\infty  (s-t) ^{\gamma -1}  e^{- \lambda (s-t) }  e^{s A} f \, ds
-   \int _0^\infty  s^{\gamma -1} e^{- \lambda s} e^{ sA} f \, ds \\ &
= 
  \int _t ^\infty  [ (s-t) ^{\gamma -1}  e^{- \lambda (s-t) }  - s^{\gamma -1} e^{- \lambda s} ]  e^{s A} f \, ds
- \int _0^t  s^{\gamma -1} e^{- \lambda s} e^{ sA} f \, ds  \\ & = I_1 + I_2 .
\end{split} 
\end{equation*} 
Clearly,
\begin{equation*} 
 \|I_2 \|_X \le   \| f \|_X \int _0^t   s^{\gamma -1} e^{- \lambda s}  ds \le  \|f\|_X  \frac {  t^\gamma} {\gamma } .
\end{equation*} 
Since $\gamma \le 1$, $ (s-t) ^{\gamma -1}  e^{- \lambda (s-t) }  \ge s^{\gamma -1} e^{- \lambda s} $, so that
\begin{equation*} 
\begin{split} 
 \| I_1 \|_X & \le    \| f \|_X \int _t ^\infty  [ (s-t) ^{\gamma -1}  e^{- \lambda (s-t) }  - s^{\gamma -1} e^{- \lambda s} ] \, ds \\ & =   \| f \|_X  \Bigl(  \int _t ^\infty   (s-t) ^{\gamma -1}  e^{- \lambda (s-t) }   \, ds -  \int _t ^\infty   s^{\gamma -1} e^{- \lambda s}  \, ds \Bigr)
 \\ & =   \| f \|_X  \Bigl(  \int _0 ^\infty  s ^{\gamma -1}  e^{- \lambda s }   \, ds -  \int _t ^\infty   s^{\gamma -1} e^{- \lambda s}  \, ds \Bigr)
  \\ & =   \| f \| _X  \int _0 ^t  s ^{\gamma -1}  e^{- \lambda s }   \, ds  \le \|f\| _X \frac {  t^\gamma} {\gamma } .
\end{split} 
\end{equation*} 
Hence~\eqref{fKTS1} follows. 

We now prove~\eqref{fKTS3}, so we assume in addition that $(e^{tA}) _{ t\ge 0 }$ is an analytic semigroup. In particular, $(e^{-t(\lambda -A)}) _{ t\ge 0 }$ is analytic, so that there exists $C$ such that
\begin{equation} \label{fKTS4} 
\|  (\lambda - A) e^{-t(\lambda -A)} \| _{ \mathcal L  ( X) } \le \frac {C} {t} 
\end{equation} 
for all $t>0$. Given $f\in X$, it follows from~\eqref{fKTS2} that
\begin{equation*} 
(\lambda - A)^{-(1- \gamma ) }f  = \frac {1} { \Gamma (1- \gamma )} \int _0^\infty  s^{- \gamma }  e^{ -s(\lambda -A)} f \, ds .
\end{equation*} 
Replacing $f$ by $ e^{-t(\lambda -A)} f$ with $t>0$, then applying $ \lambda I -A$, we obtain
\begin{equation*} 
\begin{split} 
(\lambda - A)^{ \gamma } e^{-t(\lambda -A)} f &= (\lambda - A) (\lambda - A)^{-(1- \gamma ) } e^{-t(\lambda -A)} f \\ &= \frac {1} { \Gamma (1- \gamma )} \int _0^\infty  s^{- \gamma }   (\lambda - A))  e^{- s (\lambda -A}  e^{-t(\lambda -A)} f \, ds 
 \\ &= \frac {1} { \Gamma (1- \gamma )} \int _0^\infty  s^{- \gamma }   (\lambda - A)  e^{- (s+t) (\lambda -A)}   f \, ds .
\end{split} 
\end{equation*} 
Estimate~\eqref{fKTS4} yields
\begin{equation*} 
 \| (\lambda - A)^{ \gamma } e^{-t(\lambda -A)} f \|_X \le  \frac {C} { \Gamma (1- \gamma )} \int _0^\infty   \frac {s^{- \gamma }} {s+t}ds =  \frac {C t^{-\gamma }} { \Gamma (1- \gamma )} \int _0^\infty   \frac {s^{- \gamma }} {s+1}ds ,
\end{equation*} 
from which~\eqref{fKTS3} easily follows.
\end{proof}

\begin{proof} [Proof of Lemma~$\ref{eRHE1}$]
Given $0\le t\le t+\tau \le T$,
\begin{equation*} 
\begin{split} 
v(t+\tau )- v(t) & =  \int _0^{t+\tau } e^{ (t+ \tau  -s) A} f(s) \, ds -  \int _0^t e^{ (t -s) A} f(s) \, ds \\
& =  (e^{\tau A} - I) \int _0^ t  e^{ (t  -s) A} f(s) \, ds +   \int _0^\tau  e^{s A} f(t+ \tau - s) \, ds
\\ &= I_1 + I_2 .
\end{split} 
\end{equation*} 
Clearly,
\begin{equation*} 
 \| I_2 \| \le  \tau   \| f \| _{ L^\infty ((0,T), X) }. 
\end{equation*} 
Moreover, given $0< \gamma <1$, 
\begin{equation*} 
I_1 =  (e^{\tau A} - I) (I - A)^{-\gamma } \int _0^ t (I - A)^{ \gamma }  e^{ (t  -s) A} f(s) \, ds .
\end{equation*} 
It follows from~\eqref{fKTS1} and~\eqref{fKTS3} that
\begin{equation*} 
\begin{split} 
 \| I_1 \| & \le \frac {2 \tau ^\gamma } {\Gamma ( \gamma +1)} \int _0^t  \| (I - A)^{ \gamma }  e^{ (t  -s) A} f(s)  \| \, ds \\ & \le C \tau ^\gamma   \int _0^t (t-s)^{-\gamma }  \| f(s) \|_X \, ds
  \\ & \le C \tau ^\gamma t^{1-\gamma } \| f \| _{ L^\infty ((0,T), X) }. 
\end{split} 
\end{equation*} 
This completes the proof.
\end{proof} 

\begin{lem} \label{eKTS2} 
Let $X$  be either $ \Cz$ or $\Cbu$. 
If $f\in X$ and $\Delta f\in X$, then $\nabla f\in X$ and
\begin{equation} 
 \| \nabla f \| _{ L^\infty  } \le C(  \| f \| _{ L^\infty  } +  \| \Delta f\| _{ L^\infty  }).
\end{equation} 
\end{lem} 

\begin{proof} 
We write 
\begin{equation*} 
f= (I-\Delta ) ^{-1} ( f - \Delta f).
\end{equation*} 
By formula~\eqref{fKTS2}, this means
\begin{equation} 
f =   \int _0^\infty   e^{-  s} e^{ s \Delta } (f - \Delta f ) \, ds
\end{equation} 
Therefore, we need only show that the map
\begin{equation} 
u \mapsto  \int _0^\infty   e^{-  s} \nabla e^{ s \Delta } u \, ds
\end{equation} 
is a bounded operator on $X$. This is clear, since 
\begin{equation} 
\nabla e^{ s \Delta } u = \nabla K ( s, \cdot ) \star u 
\end{equation} 
and the map $s \mapsto  \nabla K ( s, \cdot )$ is continuous $(0,\infty ) \to L^1 (\R^N ) $ with the estimate
\begin{equation} 
 \| \nabla K ( s , \cdot ) \| _{ L^1 } \le s^{- \frac {1} {2}} \pi ^{-\frac {N} {2}} \int  _{ \R^N  }  |x| e^{-  |x|^2} dx 
\end{equation} 
This completes the proof.
\end{proof} 

\begin{proof} [Proof of Theorem~$\ref{eRHE0}$]
We let
\begin{equation*} 
u (t) = e^{t\Delta }\DI + v(t),
\end{equation*} 
where
\begin{equation*} 
v(t)=    \int _0^t e^{ (t -s) \Delta }  |u(s)|^\alpha u(s)  \, ds .
\end{equation*} 
It follows from Lemma~\ref{eRHE1} that $v:[0 ,T] \to X$ is H\"older continuous. 
By analyticity (Lemma~\ref{eHRN1:b1}), $t\mapsto e^{t\Delta }\DI $ is H\"older continuous $[\varepsilon ,T] \to X$ for every $0<\varepsilon < T$. Hence so is $u$, and therefore also $ |u|^\alpha u$. 
Furthermore, $w(t) = u(t+\varepsilon )$ for $0\le t\le T-\varepsilon $ satisfies
\begin{equation} \label{fRHE5} 
w(t)= e^{t\Delta } u(\varepsilon ) + \int _0^t  e^{ (t -s) \Delta }  |w(s)|^\alpha w(s)  \, ds .
\end{equation} 
Since $ |w|^\alpha w$ is H\"older continuous $[0,T-\varepsilon ] \to X$, it follows from~\cite[Chapter~9, Theorem~1.27]{Kato} that  $w\in C^1 ( (0,T-\varepsilon ), X)$, $\Delta w\in C ( (0,T-\varepsilon ), X)$, and $w_t= \Delta w+  |w|^\alpha w$ for all $0<t<T-\varepsilon $.
Since $0<\varepsilon <T$ is arbitrary, this proves the first part of the theorem.

Next, it follows from Lemma~\ref{eKTS2} that $\nabla u \in C ((0, T), X)$, so that $\nabla (  |u|^\alpha u )\in C((0,T), X)$. 
Therefore, we may take the gradient of~\eqref{fRHE5}, and we obtain
\begin{equation} \label{fRHE6} 
\nabla w(t)= e^{t\Delta } \nabla u(\varepsilon ) + \int _0^t  e^{ (t -s) \Delta } \nabla ( |w(s)|^\alpha w(s) ) \, ds .
\end{equation} 
As above, we deduce using Lemma~\ref{eRHE1} and analyticity that $\nabla u$ is H\"older continuous $[\varepsilon ,T] \to X$ for every $0<\varepsilon < T$. Hence so is $\nabla ( |u|^\alpha u) $. 
Still as above, it follows from~\cite[Chapter~9, Theorem~1.27]{Kato} that  $\nabla u \in C^1 ( (0,T ), X)$ and $\Delta \nabla u\in C ( (0,T ), X)$. 
Applying again Lemma~\ref{eKTS2}, we deduce that every space derivative of $u$ of order  two  is in $C((0,T), X)$.
\end{proof}  

We now consider the case of the heat equation on a bounded, smooth domain. 
The following result is standard (see e.g.~\cite[Theorem~15.2]{QuittnerS}), but we give here a rather simple proof similar to the proof of Theorem~\ref{eRHE0}. 

\begin{thm} \label{eRHE2} 
Let $\Omega \subset \R^N $ be a bounded domain with boundary of class $C^{2, \mu }$ for some $\mu >0$,  let $X= \Czo$, and $(e^{t \Delta _\Omega }) _{ t\ge 0 }$ the heat semigroup on $X$. 
Let $\alpha >0$, $\DI \in  X$, $T>0$, and suppose $u\in C( [0,T], X)$ satisfies
\begin{equation}  \label{eRHE2:1} 
u(t)= e^{t \Delta _\Omega } \DI +  \int _0^t e^{ (t -s) \Delta _\Omega }  |u(s)|^\alpha u(s)  \, ds
\end{equation} 
for all $0\le t\le T$. It follows that $u \in C^1 ( (0,T), X)$, $\Delta u \in C ( (0,T), X)$, and $u_t= \Delta u+  |u|^\alpha u$ for all $0<t<T$. 
In addition, given any $\omega \subset \subset \Omega $, it follows that $\nabla u\in C^1 ((0, T), C( \overline{\omega })) $ and $u\in C((0, T), C^2 ( \overline{\omega } ))$.
\end{thm} 

\begin{proof} 
We recall that $(e^{t \Delta _\Omega }) _{ t\ge 0 }$ is an analytic semigroup on $X$, see~\cite[Theorem~5]{Stewart}. 
Therefore, we can argue as in the proof of Theorem~\ref{eRHE0}, which establishes the first part of the result. 
To prove the local regularity, we argue as follows. 
Recall that the domain of the Laplacian on $\Czo$ is $Y= \{ u\in X; \, \Delta u\in X \}$, which we equip with the graph norm (so that $Y$ is a Banach space). We also consider  $Z= C(\Omega ) \cap L^\infty (\Omega ) $ equipped with the sup norm.  It follows from~\cite[Theorem~1]{Stewart} that if $u\in Y$, then $\nabla u\in Z $ and 
\begin{equation} \label{eRHE2:2} 
 \| \nabla u \| _{ L^\infty  } \le C (  \| u \| _{ L^\infty  } +  \| \Delta u\| _{ L^\infty  }).
\end{equation} 
Next, we observe that $u(t) \in Y$ for all $0< t\le T$. Therefore, after possibly a time translation, we may assume that $\DI\in Y$ and $u\in C^1([0, T], X)$. We write $u(t) = w(t) + v (t) $ where $w(t) = e^{t\Delta _\Omega } \DI$ and
\begin{equation}  \label{eRHE2:3} 
v (t)=   \int _0^t e^{ s \Delta _\Omega }  |u(t-s)|^\alpha u(t- s)  \, ds .
\end{equation} 
Since $u\in C^1([0, T], X)$, it follows easily that $ |u|^\alpha u\in C^1([0, T], X)$ and $\frac {d} {dt} (  |u|^\alpha u )= (\alpha +1)  |u|^\alpha \frac {du} {dt}$. Therefore, we may differentiate~\eqref{eRHE2:3} with respect to $t$, and we obtain that 
\begin{equation}  \label{eRHE2:3:b1} 
v'  (t)=  (\alpha +1)  \int _0^t e^{ s \Delta _\Omega }  |u(t-s)|^\alpha u' (t- s)  \, ds .
\end{equation} 
It follows from Lemma~\ref{eRHE1} that $v' :[0 ,T] \to X$ is H\"older continuous. 
Since $\DI\in Y$, $w: [0,T] \to X$ is $C^1$, hence also H\"older continuous. Therefore, so is $u'$, so that  $ |u|^\alpha u'$ is also H\"older continuous. It now follows from~\cite[Chapter~9, Theorem~1.27]{Kato} that  $ v ' \in C^1 ( (0,T ), X)$ and $\Delta v ' \in C ( (0,T ), X)$. By analyticity, the same is true for $w'$, hence for $u'$. 
In particular, $u\in C^1 ((0,T), Y)$, and we deduce from~\eqref{eRHE2:2} that $\nabla u\in C^1((0,T), Z )$. 
Therefore, $\nabla \Delta u=\nabla (u_t-|u|^\alpha u) \in C((0,T), Z)$.
Given $N<p<\infty $, we deduce that $u\in C((0, T), W^{1,p} (\Omega ) )$ and $\Delta u\in C((0, T), W^{1,p} (\Omega ) )$. By local regularity (see e.g.~\cite[Theorem~17.1.3]{Hormander}), it follows that if $\omega \subset \subset \Omega $, then $u\in C((0,T), W^{3, p}  (\omega ) )$, hence $u\in C((0,T), C^2 ( \overline{\omega}  ) )$ since $p> N$. This completes the proof.
\end{proof}

\end{document}